\documentclass[11pt]{amsart}
\usepackage{amsmath,amsfonts,amsthm,amssymb,amscd,stmaryrd,url}
\usepackage{color}
\usepackage[top=3.0cm,left=3.0cm,right=3.0cm,bottom=3.0cm]{geometry}
\usepackage[textwidth=30mm]{todonotes}
\usepackage{ulem} 
\usepackage[utf8]{inputenc}
\usepackage{hyperref}
\hypersetup{
  colorlinks   = true, 
  urlcolor     = blue, 
  linkcolor    = blue, 
  citecolor   = purple 
}

\usepackage{caption}
\usepackage{cases}
\usepackage{epsfig}
\usepackage{float,mathtools}
\usepackage[stable]{footmisc}
\usepackage{graphicx}
\usepackage{hyperref}
\usepackage{latexsym}
\usepackage{mathrsfs}
\usepackage{multirow}
\usepackage{textcomp} 
\usepackage{textgreek}
\usepackage{titletoc}
\usepackage{xcolor}

\DeclareMathOperator{\li}{li}
\DeclareMathOperator{\Li}{Li}

\def\N{\mathbb{N}}

\newtheorem{theorem}{Theorem}
\newtheorem{corollary}[theorem]{Corollary}
\newtheorem{lemma}[theorem]{Lemma}
\newtheorem{prop}[theorem]{Proposition}
\theoremstyle{definition}
\newtheorem{definition}[theorem]{Definition}
\newtheorem{remark}[theorem]{Remark}

\begin{document}

\title{Sharper bounds for the error term in the Prime Number Theorem}

\author{Andrew Fiori, Habiba Kadiri, Joshua Swidinsky}
\address{Department of Mathematics and Statistics, University of Lethbridge, Canada}
\email{andrew.fiori@uleth.ca, habiba.kadiri@uleth.ca, jds26@sfu.ca}
\thanks{Andrew Fiori graciously acknowledges both the University of Lethbridge for their financial support (start-up grant and ULRF) as well as the support of NSERC Discovery Grant RGPIN-2020-05316.
Habiba Kadiri acknowledges both the University of Lethbridge for their financial support (ULRF) as well as the support of NSERC Discovery Grant RGPIN-2020-06731.
Joshua Swidinsky acknowledges receiving financial support through an NSERC USRA and from the University of Lethbridge for this work.
Computations performed in this work were conducted using infrastructure supported by WestGrid (\url{www.westgrid.ca}) and Compute Canada (\url{www.computecanada.ca}).
}






\begin{abstract}
We obtain bounds for the error term in the prime number theorem of the form
\[\left| \pi(x) - \Li(x) \right| \leq 9.2211\, x\sqrt{\log(x)} \exp\left( -0.8476 \sqrt{\log(x)} \right)\ \text{for all} \ x\ge 2,\]
as well as other classical forms, improving upon the various constants and ranges compared to those in the literature.
The strength and originality of our methods come from leveraging numerical results for small $x$ in order to improve both the asymptotic and numerical bounds one obtains.
We develop algorithms and formulas optimizing the conversion of both asymptotic and explicit numerical bounds from the prime counting function $\psi(x)$ to both $\theta(x)$ and $\pi(x)$.
\end{abstract}

\maketitle


\section{Introduction}

The aim of this work is to provide explicit and efficient methods to convert between bounds on the error terms for various prime counting functions. 
Ultimately it provides new explicit and unconditional bounds on the error term in the Prime Number Theorem
\begin{equation}\label{def-Epi}
E_{\pi}(x)=\left|\frac{\pi(x) - \Li(x)}{x /\log x}\right|,
\end{equation}
where $\Li(x)=\int_2^x\frac{dt}{\log (t)}$.
There is a long history of work in this area, dating back to works of Rosser and Schoenfeld (see \cite{Rosser, RosserSchoenfeld1,RosserSchoenfeld2,Schoenfeld}). 
We introduce the notation for the error terms associated to the Chebyshev functions
\begin{equation}\label{def-Etheta-Epsi}
E_{\theta}(x)=\left|\frac{\theta(x) -x}{x}\right|, \ \text{and}\ 
E_{\psi}(x)=\left|\frac{\psi(x) - x}{x}\right|.
\end{equation}
We are interested in two types of explicit bounds: asymptotic formulas and numerical values.
We develop algorithms and formulas to provide optimized conversions of both type.
Finally, we leverage our numerical results to improve the quality of the asymptotic formulas we obtain. More precisely, the key results of this paper are the following: 
\begin{itemize}
\item Effective conversions for asymptotic bounds from $E_{\psi}(x)$ to $E_{\theta}(x)$  are described in Proposition \ref{epsilon_asymp_theta_prop} and are applied to give a bound for $E_{\theta}(x)$ in Corollary \ref{rem-always-theta}. 
Effective conversions for asymptotic bounds from $E_{\theta}(x)$ to $E_{\pi}(x)$ are described in Theorem \ref{prop_asym_pi} and Corollary \ref{asymp-bnd-psi2pi}. 
The final asymptotic bounds we obtain on $\pi(x)$ are given in Corollary \ref{epsilon_asymp_pi_explicit}:
\begin{equation}\label{ineq-pi-exp}
\left| \pi(x) - \Li(x) \right| \leq 9.2211\, x\sqrt{\log(x)} \exp\left( -0.8476 \sqrt{\log(x)} \right), \ \text{for all}\ x\ge 2. 
\end{equation}
\item A mechanism to convert numerical bounds from $E_{\psi}(x)$ to $E_{\theta}(x)$ are found in Proposition \ref{prop:numpsitotheta}, and then from $E_{\theta}(x)$ to $E_{\pi}(x)$ in Theorem \ref{prop_num_pi}.
 Tables \ref{table:theta-small}, \ref{table:theta-med}, and \ref{table:theta} list numerical bounds which work for both $E_{\psi}(x)$ and $E_{\theta}(x)$. Finally, Table \ref{tab:pi} list numerical bounds for $E_{\pi}(x)$ as a consequence.
More detailed versions of these tables will be posted online in \cite{FKS3}. 
The complete data is available upon request to the authors.  
Though such tables of values can be unwieldy they can be useful for certain technical results, see for example \cite{Axler,Joh22}.
In addition to the tables we also give a few other bounds for $E_{\pi}(x)$ which may be of interest in Corollaries \ref{cor:alt1}, \ref{cor:alt2}, and \ref{cor:weak}, such as
\begin{equation}\label{ineq-pi-log}
\left| \pi(x) - \Li(x) \right| \leq 0.4298\frac{x}{\log(x)}, \ \text{for all}\ x\ge 2.
\end{equation}
\end{itemize}
We remind the reader that, depending on the size of $x$, strong bounds for $E_{\psi}(x)$ can be obtained from \cite{Bu16}, \cite{Bu18}, \cite{FKS}, and \cite{JohYan22}.

\subsection{Notation}
Because our bounds all derive from information about zeros of the $\zeta$-function, we shall need to make reference to values of $R$ so that  
$\zeta(s)$ has no zeros in the region $\operatorname{Re}(s) \geq 1 - \frac{1}{R\log \operatorname{Im}(s)}$ for $\operatorname{Im}(s) \geq 3$.
Throughout this paper $R$ denotes any fixed value with this property.
Where specific numerics are given they are relative to \[R=5.5666305,\] the validity of which follows from the work of \cite[Theorem 1 and Section 6.1]{TrudMos15}.
Note that since this paper was submitted, Mossighoff, Trudgian, and Yang \cite{MoTrYa22} have announced a reduced value of $5.558691$ for $R$. The numerical values presented in the current article do not reflect this improvement. 
For clarity, we recall that
\[ \li(x) = \int_0^{x} \frac{1}{\log(t)}{\rm dt} ,\ \text{and}\  \Li(x) =  \int_2^{x} \frac{1}{\log(t)}{\rm dt}. \]
Since $\li(x)-\Li(x)=\li(2)\approx 1.04516,$ there is no asymptotic difference in comparing $\pi(x)$ to either $\li(x)$ or $\Li(x)$.

\subsection{Comparisons to Other Work}
In addition to providing very tight bounds for the error terms, this article also provides methods to convert results of a certain kind into another.
 For instance, this complements methods given in \cite{BKLNW21} and \cite{FKS} to obtain bounds for $\theta(x)$.

We note that our Corollary \ref{epsilon_asymp_pi_explicit} improves Platt and Trudgian's \cite[Corollary 2]{PlaTru21ET}:
\[ | \pi(x)-\li(x)| \leq 235 x\left(\log(x)\right)^{0.52} \exp(-0.8\sqrt{\log(x)} ) \ \text{for all}\ x \ge \exp(2\,000),\]
and the more recent asymptotic results of Johnston and Yang \cite[Corollary 1.3]{JohYan22}:
\[ | \pi(x)-\li(x)| \leq 9.59x\left(\log(x)\right)^{0.515} \exp(-0.8274\sqrt{\log(x)} ) \ \text{for all}\ x \ge 2. \]
Our  Corollary \ref{epsilon_asymp_pi_explicit} is both asymptotically and point-wise better than the above result.
However, by using the Korobov-Vinogradov's zero-free region, they also obtain \cite[Theorem 1.4]{JohYan22}:
\begin{equation}\label{bnd-JW} 
| \pi(x)-\li(x)| \leq 0.028x\left(\log(x)\right)^{0.801} \exp\left(-0.1853 \frac{\left(\log(x)\right)^{3/5}}{\left(\log \log(x)\right)^{1/5}} \right) \ \text{for all}\ x \ge 23. 
\end{equation}
This later result is asymptotically, and ultimately pointwise, better than ours.
For instance, comparing directly Corollary \ref{epsilon_asymp_pi_explicit} and \eqref{bnd-JW}, the latter becomes sharper for all $x\ge \exp(1.826\cdot 10^9)$ 
We note that both \cite{FKS} and \cite{JohYan22} provide more refined results for the admissible values $A_{\pi},B,C,x_0$ and the resulting numerical bound $\varepsilon_{\pi,num}(x_0)$ as $x_0$ gets larger. 
For instance, in \cite[Table 1]{JohYan22}, $x_0=\exp(100\,000)$ is the first displayed value where Johnston and Yang's numerical bound for $\pi$ becomes smaller than ours. Namely, they find $\varepsilon_{\pi,num}(x_0) \le 9.12 \cdot 10^{-111}$ when we find $\varepsilon_{\pi,num}(x_0) \le 7.79 \cdot 10^{-109}$.
Taken together, the results of this current article for smaller values $x$ and their results for larger values $x$ represent the current state of the art for explicit and unconditional bounds on the error terms in the prime number theorem.
Combining the different optimizations from \cite{JohYan22} with those from \cite{FKS}, along with those in the present article, is a natural future project. As it would require redoing all of the proofs of \cite{FKS}, it is as such a non-trivial undertaking outside the scope of this article.

\subsection{Our asymptotic bounds}

The asymptotic bounds we will study for $\psi$, $\theta$, and $\pi$ all take a specific form:
\begin{definition}\label{epsilon_asymp_theta_def}
With $R$ fixed, and for constants $A,B,C$, and for $\diamond$ one of $\psi$, $\theta$,  or $\pi$, we define
We say that $A,\,B,\,C,\,x_0$ give an admissible asymptotic bound for $\diamond$ if we have that, for all $x\ge x_0$,
 \begin{equation}
 \label{bnd-asymp-E-eps}
E_{\diamond}(x)  \leq \varepsilon_{\diamond,asymp, A,B,C}(x), 
 \ \text{where}\    \varepsilon_{\diamond, asymp, A,B,C}(x) =  A\left( \frac{\log(x)}{R} \right)^B \exp\left( -C \sqrt{\frac{\log(x)}{R}} \right).
 \end{equation}
 For simplicity, when there is no confusion, we will use the notation  $\varepsilon_{\diamond,asymp}$ in place of $\varepsilon_{\diamond, asymp, A,B,C}$. 
  \end{definition}
Note that we are comparing $|\pi(x)-\Li(x)|$ against $\frac{x}{\log x}$. One can obtain similar results with $\Li(x)$ however the latter seems to be more in use in applications.
 \begin{remark}
 Many admissible asymptotic bounds for $E_{\psi}$ have been proven.
For example by \cite[Corollary 1.3]{FKS}, when using $R=5.5666305$, one has an admissible asymptotic bound \eqref{bnd-asymp-E-eps} for $E_{\psi}(x)$ with 
\[  A_{\psi}= 121.096, \ B = 3/2 ,\ C = 2,\ \text{and}\  x_0 =2.\]
 Formulas from \cite{FKS} readily allow one to 
 compute admissible values of $x_0$ and $ A_{\psi}$, with $B = 3/2$, and $C = 2$ for any other fixed $R < 5.5666305$.
 \end{remark}
The first main result of this paper is an explicit method to convert admissible asymptotic bounds for $\theta$ (or $\psi$ by also using Proposition \ref{epsilon_asymp_theta_prop}) into bounds for $\pi$.
\begin{theorem}\label{prop_asym_pi}
Let $B \geq \max(\frac{3}{2}, 1+\frac{C^2}{16R})$. 
Let $x_0>0$ verifying $\pi(x_0)$ and $\theta(x_0)$ are computable. 
Assume $ A_{\theta},B,C,x_0$ give an admissible asymptotic bound \eqref{bnd-asymp-E-eps} for $E_{\theta}(x)$, and let \begin{equation}
\label{def-x1}
x_1 \geq \max\left\{x_0, \exp\Big(\big(1+\frac{C}{2\sqrt{R}}\big)^2\Big)\right\}.
\end{equation}
We have 
\begin{equation}
\label{epsilon_pi_asymp_def}
E_{\pi}(x)\leq \varepsilon_{\pi,asymp}(x) = \varepsilon_{\theta,asymp}(x)\left( 1 + \mu_{asymp}(x_0, x_1) \right),\ \text{for all}\ x\ge x_1,
\end{equation}
where
\begin{multline}
 \label{mu_asymp_def}
\mu_{asymp}(x_0, x_1)
\\    =   \frac{x_0\log(x_1)}{\varepsilon_{\theta,asymp}(x_1)x_1\log(x_0)}  \left| \frac{\pi(x_0) - \Li(x_0) }{ x_0 /\log(x_0)}
    -  \frac{\theta(x_0) - x_0}{x_0} \right| 
 + \frac{2 D_{+}\big( \sqrt{\log(x_1)} - \frac{C}{2\sqrt{R}} \big)}{\sqrt{\log(x_1)}},
\end{multline}
and $D_{+}(x)$ is the Dawson function defined in \eqref{Dawson}.
Defining
$
A_{\pi}=\left( 1 + \mu_{asymp}(x_0, x_1) \right) A_{\theta},
$ 
we have that $ A_{\pi},B,C,x_1$ give an admissible asymptotic bound \eqref{bnd-asymp-E-eps} for $E_{\pi}(x)$.
\end{theorem}
We refer the reader to Section \ref{section:asymptotic_bounds} for proofs, to Proposition \ref{epsilon_asymp_theta_prop} for a formula to relate $\varepsilon_{\theta,asymp}(x)$ with $\varepsilon_{\psi,asymp}(x)$, and to Section \ref{Section-Cor} 
 for Corollaries illustrating the use of the Theorem.
Note that one can still obtain a formula for admissible asymptotic bounds on $\pi$ while removing the restriction imposed on $B$ (see Remark \ref{rem-pi2theta}).
\begin{remark}
It is important to note that Theorem \ref{prop_asym_pi} gives a bound for $\pi(x)$ for $x\ge x_1$, but that the assumption is on $\theta(x)$ for $x\ge x_0$, where $x_0$ and $x_1$ do not necessarily coincide. 
Also, it is useful if $\pi(x_0)$ and $\theta(x_0)$ are computable so that we can evaluate the formula $  \pi(x_0) - \Li(x_0) - \frac{\theta(x_0) - x_0}{\log(x_0)} $.
We note that 
\[ \frac{\pi(x_0) - \Li(x_0) }{ x_0 /\log(x_0)}
    -  \frac{\theta(x_0) - x_0}{x_0} = 0 \ \text{for}\ x_0=40.787732519\ldots.\]
This simplifies the result for bounds that are admissible with this $x_0$.
However, because admissible values for $ A_{\theta}$ tend to decrease with $x_0$, and because Theorem \ref{prop_asym_pi} requires the value of $ A_{\theta}$ be admissible at $x_0$ it can still be useful in this theorem to estimate this quantity for larger $x_0$.
To this end the work of Stable \cite{Staple} gives many values of $\pi(x_0)$ whereas the work of Dusart  \cite{Dusart} gives many values for $\theta(x_0)$.
For example for $x_0=10^{15}$ we have 
\begin{align*}
  \theta(10^{15}) &= 999999965752660.939840\ldots,\\
  \pi(10^{15}) &= 29844570422669, \text{ and}\\
  \Li(10^{15}) &= 29844571475286.535901\ldots
  \end{align*}
and thus find  
\[ \frac{\pi(x_0) - \Li(x_0) }{ x_0 /\log(x_0)}
    -  \frac{\theta(x_0) - x_0}{x_0} = (-2.1087826\ldots )\cdot 10^{-9}. \]  
This specific computation will be used when we produce numerical results as displayed in Table \ref{tab:pi} using Theorem \ref{prop_num_pi}.

To apply Theorem \ref{prop_asym_pi} with a value of $ A_{\theta}$ which is only admissible for very large $x_0$, for instance those from \cite[Table 6]{FKS} or \cite[Table 1]{JohYan22},
 one may bound  $ \frac{\pi(x_0) - \Li(x_0) }{ x_0 /\log(x_0)}   -  \frac{\theta(x_0) - x_0}{x_0}$ by using the numerical estimates on $\pi(x_0)-\Li(x_0)$ and $\theta(x_0)-x_0$ provided in Tables \ref{table:theta-small}, \ref{table:theta-med}, \ref{table:theta}, and \ref{tab:pi}. 
To illustrate how this is done notice that \cite[Table 1]{JohYan22} gives 
\[A_{\theta}=23.14,\,B=1.503,\,C=2.0429\ldots,\, x_0=e^{10^5}\] are admissible asymptotic bounds for $\theta$. 
From Table \ref{table:theta} we have, with  $x_0=e^{10^5}$ that $\left|{\frac{\theta(x_0) - x_0}{x_0}}\right| \leq  7.7824\cdot 10^{-109}$ and from Table \ref{tab:pi} that $ \left|{\frac{\pi(x_0) - \Li(x_0) }{ x_0 /\log(x_0)}}\right| \leq  7.7825\cdot 10^{-109}$.
We deduce that  $ \left|{ \frac{\pi(x_0) - \Li(x_0) }{ x_0 /\log(x_0)} -  \frac{\theta(x_0) - x_0}{x_0}}\right| \leq 1.56849\cdot 10^{-108 } $ and hence, for the above $A_\theta,\,B,\,C$ we have $\mu_{asymp}(e^{10^5},e^{10^5}) \leq  2252.31$ so that
\[ A_{\pi}=  52141.6,\; B=1.503,\,C=2.0429\ldots, \ x_1=e^{10^5}\]
are admissible asymptotic bounds for $\pi(x)$ for $x \geq x_1=e^{10^5}$.
Note that we can improve $\mu_{asymp}$, and consequently $A_{\pi}$, at the cost of increasing $x_1$, for example $\mu_{asymp}(e^{10^5},e^{100016}) \leq 2.66336 \cdot 10^{-4}$ so that
\[A_{\pi}=23.15,\; B=1.503,\,C=2.0429\ldots,\, x_1=e^{100016}\] are admissible asymptotic bounds for $\pi(x)$ for $x \geq x_1=e^{100016}$. 
\end{remark}

\subsection{Our numerical bounds}
We set up our notation:
\begin{definition}\label{epsilon_medium_theta_def}
Let $x_0 > 2$. We say a (step) function $\varepsilon_{\diamond,num}(x_0)$ gives an admissible numerical bound for $E_{\diamond}(x)$ if
 \begin{equation}
 \label{bnd-num-E-eps}
 E_{\diamond}(x)  \leq \varepsilon_{\diamond,num}(x_0), \ \text{for all}\ x\ge x_0.
 \end{equation}
\end{definition}

The second main result of this paper is an explicit method to convert admissible numerical bounds for $\theta$ (or $\psi$ by also using Proposition \ref{prop:numpsitotheta})  into bounds for $\pi$.
\begin{theorem}\label{prop_num_pi}
Let $x_0>0$ be chosen such that $\pi(x_0)$ and $\theta(x_0)$ are computable, and let $x_1 \ge {\rm \max}(x_0,14)$. 
Let $\{ b_i \}_{i=1}^{N}$ be a finite partition of the interval $[\log(x_0),\log(x_1)]$, with $b_1=\log(x_0)$ and $b_N=\log(x_1)$,  and suppose that $\varepsilon_{\theta,num}(x)$ gives computable admissible numerical bounds for $x=\exp(b_i)$, for each $i=1,\ldots ,N$.

For $x_1 \leq x_2\leq x_1\log(x_1)$, we define
\begin{equation}\label{mu_num_defx1x2}
\begin{split}
    \mu_{num}(x_0, x_1,x_2) 
    & = \frac{x_0\log(x_1)}{\varepsilon_{\theta,num}(x_1)x_1\log(x_0)}  \left| \frac{\pi(x_0) - \Li(x_0) }{ x_0 /\log(x_0)}
    -  \frac{\theta(x_0) - x_0}{x_0} \right|\\ 
    & \quad + \frac{\log(x_1)}{\varepsilon_{\theta,num}(x_1)x_1}\sum_{i=1}^{N-1} \varepsilon_{\theta,num}(e^{b_i})\left( \Li(e^{b_{i+1}}) - \Li(e^{b_i}) + \frac{e^{b_i}}{b_i} - \frac{e^{b_{i+1}}}{b_{i+1}} \right) \\
   & \quad + \frac{\log(x_2)}{x_2} \left(  \Li(x_2) - \frac{x_2}{\log(x_2)}  - \Li(x_1) +\frac{x_1}{\log(x_1)} \right),
\end{split}
\end{equation}
and, for $x_2>x_1(\log(x_1))$, including $x_2=\infty$, we define
\begin{equation}\label{mu_num_def}
\begin{split}
    \mu_{num}(x_0, x_1,x_2) 
    & = \frac{x_0\log(x_1)}{\varepsilon_{\theta,num}(x_1)x_1\log(x_0)}  \left| \frac{\pi(x_0) - \Li(x_0) }{ x_0 /\log(x_0)}
    -  \frac{\theta(x_0) - x_0}{x_0} \right|\\ 
    & \quad + \frac{\log(x_1)}{\varepsilon_{\theta,num}(x_1)x_1}\sum_{i=1}^{N-1} \varepsilon_{\theta,num}(e^{b_i})\left( \Li(e^{b_{i+1}}) - \Li(e^{b_i}) + \frac{e^{b_i}}{b_i} - \frac{e^{b_{i+1}}}{b_{i+1}} \right) \\
    & \quad + \frac{1}{\log(x_1) + \log(\log(x_1))-1}.    
\end{split}
\end{equation}
Then for $x_1\leq x \leq x_2$ we have
\begin{equation}\label{epsilon_pi_mediumx1x2}
E_{\pi}(x) \leq \varepsilon_{\pi,num}(x_1,x_2) = \varepsilon_{\theta,num}(x_1)(1 + \mu_{num}(x_0,x_1,x_2)).
\end{equation}
\end{theorem}

We refer the reader to Section \ref{section:numeric_bounds} for a proof of Theorem \ref{prop_num_pi}, together with a formula relating $\varepsilon_{\psi,num}$ to $\varepsilon_{\theta,num}$ (see Proposition \ref{prop:numpsitotheta}). 
The reader will also find several simpler forms of bounds for $\pi(x)$ in Corollaries \ref{epsilon_asymp_pi_explicit}, \ref{cor:alt1}, \ref{cor:alt2}, and \ref{cor:weak}. 

\begin{remark}\label{rem:consecutive}
From the proof we will find that the bound from \eqref{mu_num_def} is in fact applicable for all $x>x_1$ and that the bound \eqref{mu_num_defx1x2} is valid. In addition it is strictly better than  \eqref{mu_num_def}, provided
\[  \frac{d}{dx}\frac{\log(x)}{x} \left(  \Li(x) - \frac{x}{\log(x)}  - \Li(x_1) +\frac{x_1}{\log(x_1)} \right) \Bigg|_{x=x_2} > 0.\]
So although \eqref{mu_num_def} is easier to use directly, \eqref{mu_num_defx1x2} allows for strictly better bounds.
For instance, we have $\mu_{num}(10^{15},e^{100},\infty)=0.0197\ldots$, while $\mu_{num}(10^{15},e^{100},e^{101})=0.0165\ldots$ and $\mu_{num}(10^{15},e^{100},e^{100.1})= 0.0111\ldots$.
One can recover bounds valid for all $x>x_1$ from the bounds \eqref{mu_num_defx1x2}  by considering the maximum of $\varepsilon_{\psi,num}(x_i,x_{i+1})$ over consecutive intervals. This idea ultimately provides better results than  \eqref{mu_num_def}, and is made more precise in the following corollary.
\end{remark}
\begin{corollary}\label{cor_num_pi}
Consider a set $\mathcal{B'}=\{b_i'\}_{i=1}^{M}$ of finite subdivisions of $[\log(x_1),\infty)$, with $b_1'=\log(x_1)$ and $b_{M}'=\infty$.
We define
\begin{equation}\label{def-eps0B}
\varepsilon_{\pi,num}(x_1) = \max_{1\le i\le M-1} \left\{\varepsilon_{\pi,num}(\exp(b_i'),\exp( b_{i+1}')) \right\} 
\end{equation}
where $\varepsilon_{\pi,num}(\exp(b_i'),\exp( b_{i+1}'))$ is given by \eqref{epsilon_pi_mediumx1x2}.
Then, 
\begin{equation}
\label{epsilon_pi_num_def}
E_{\pi}(x) \leq \varepsilon_{\pi,num}(x_1) ,\ \text{for all}\  x \ge x_1.
\end{equation}
\end{corollary}
We refer the reader to Table \ref{tab:pi} where we provide a sample of new numerical values arising from Corollary \ref{cor_num_pi} and Theorem \ref{prop_num_pi}.
%
\begin{remark}
Note that because both $ \varepsilon_{\theta,num}(\exp(b_i'))$ and $\mu_{num}(x_0,\exp(b_i'),\exp(b_{i+1}'))$ tend to decrease as $i$ increases, in practice the maximum will typically come from the ``first" interval, $[\exp(b_{1}'),\exp(b_{2}')]$. 
Moreover, because the last term in \eqref{mu_num_defx1x2}
\[\frac{ b_{i+1}'}{\exp(b_{i+1}')} \left(  \Li(\exp(b_{i+1}')) - \frac{\exp(b_{i+1}')}{ b_{i+1}' }  - \Li(\exp(b_{i}')) +\frac{\exp(b_{i}')}{ b_{i}'} \right)\]
becomes smaller when $b_{i+1}'$ is close to $b_{i}'$, it is advantageous to take small intervals.
\end{remark}

\subsection{Acknowledgements}
The authors thank the referee for their careful reading of our manuscript.

\section{An Explicit Version of the Asymptotic Bound for $\pi(x)$}\label{section:asymptotic_bounds}

The shape of the error term function can be defined as 
\begin{equation}\label{def-g}
 g(a,b,c,x) = x^{-a}(\log(x))^b \exp(c\sqrt{\log(x)}) ,
 \end{equation}
    for some real numbers $a,b,c$ with $a\ge 0$ and $c>0$.
Lemma \ref{decreasing_func} gives some facts about the behaviour of this function. 

The first step to prove Theorem \ref{prop_asym_pi} is to use the identity derived from the Stieljtes integral (see \cite[Equation (4.17)]{RosserSchoenfeld1}) to connect $\pi(x)$ to $\theta(x)$:
\begin{equation}\label{Stieltjes}
    \left(\pi(x) - \Li(x)\right) - \left(\pi(x_0) - \Li(x_0)\right) = \frac{\theta(x)-x}{\log(x)} - \frac{\theta(x_0)-x_0}{\log(x_0)} + \int_{x_0}^{x} \frac{\theta(t)-t}{t(\log(t))^2}dt.
\end{equation}

We will also need some estimates for $\psi(x)$ to inform those for $\theta(x)$ (see Proposition \ref{epsilon_asymp_theta_prop}).

The main difference between our result for admissible asymptotic bounds on $\pi$ and admissible numerical bounds on $\pi$ (as will be described in Section \ref{section:numeric_bounds}) is how we will estimate the integral
\begin{equation}\label{important_integral}
    \int_{x_0}^{x} \left| \frac{\theta(t)-t}{t(\log(t))^2} \right|dt.
\end{equation} 

\subsection{Preliminary Lemmas}
\begin{lemma}\label{decreasing_func}
Let $a,b,c$ be any real numbers, with $a\ge 0$ and $c > 0$. 
We recall the function $g(a,b,c,x) $ as defined in \eqref{def-g}:
\[ g(a,b,c,x) = x^{-a}(\log(x))^b \exp(c\sqrt{\log(x)}). \]
    \begin{itemize}
\item If $a > 0$ and $b < -\frac{c^2}{16a}$, then $g(a,b,c,x) $ decreases with $x$, 
for all $x$.
\item If $a > 0$ and $b \ge -\frac{c^2}{16a}$, then $g(a,b,c,x) $ decreases as a function of $x$, 
for all $x$ with $x> \exp\big( \big(\frac{c}{4a} + \frac1{2a}\sqrt{\frac{c^2}{4}+4ab} \big)^2\big)$. 
\item If $a=0$, then $g(0,b,c,x) $ decreases with $x$, for all $\sqrt{\log(x)}>-\frac{2b}c$.
\end{itemize}
\end{lemma}
\begin{proof}
    Differentiating $g(a,b,c,x)$ with respect to $x$, we see that
    \[ \frac{d}{dx} g(a,b,c,x) = \left( -a\log(x) + b + \frac{c}{2}\sqrt{\log(x)} \right)  x^{-a-1}(\log(x))^{b-1}\exp(c\sqrt{\log(x)}) ,\]
    which is negative when  $-au^2 + \frac{c}{2}u + b < 0$,
    where $u = \sqrt{\log(x)}$. 
If $a=0$, it is negative when $u<-\frac{2b}{c}$. 
If $a > 0$, there are two real roots only if $\frac{c^2}{4} + 4ab \geq 0$ or equivalently $ b \geq -\frac{c^2}{16a}$, and the derivative is negative for $u > \frac{\frac{c}{2} + \sqrt{\frac{c^2}{4}+4ab}}{2a}$.
Otherwise, there are no roots when $b < -\frac{c^2}{16a}$, and the derivative is always negative in this case. 
\end{proof}
An immediate corollary is the following:
\begin{corollary}\label{decreasing_func_corollary}
    If $B \geq 1 + \frac{C^2}{16R}$ then $g(1, 1-B, \frac{C}{\sqrt{R}}, x)$ is decreasing for all $x$.
\end{corollary}

To estimate the integral \eqref{important_integral}, we shall need the Dawson function 
\begin{equation}\label{Dawson}
    D_{+}(x) = e^{-x^2} \int_{0}^{x} e^{t^2}dt .
\end{equation}
It is well known that the Dawson function has a single maximum at $x \approx 0.942$, after which the function is decreasing\footnote{The Dawson function satisfies the differential equation 
$ F'(x)+2xF(x)=1$
 from which it follows that the second derivative satisfies $F''(x) = -2F(x) -2x(-2xF(x)+1)$, so that at every critical point (where we have $F(x)=\frac{1}{2x}$) we have $F''(x) = -1/x$.
 It follows that every positive critical value gives a local maximum, hence there is a unique such critical value and the function decreases after it. 
Numerically one may verify this is near $0.9241$
see \url{https://oeis.org/A133841}.}. 
We will apply this fact later.

\begin{lemma}\label{integral_asymp_bound}
Assume that $ A_{\theta}(x_0),\, B,\, C,\ x_0$ provide an admissible asymptotic bound \eqref{bnd-asymp-E-eps} for $\theta$. Then, for all $x\geq x_0$, 
    \begin{equation}\label{integral_head_bound}
        \int_{x_0}^{x} \left| \frac{\theta(t) - t}{t(\log(t))^2} \right| dt  \leq 
        \frac{2 A_{\theta}}{R^B}xm(x_0,x)\exp\left( - C\sqrt{\frac{\log(x)}{R}} \right)D_{+}\left( \sqrt{\log(x)} - \frac{C}{2\sqrt{R}} \right)
    \end{equation}
    where 
    \begin{equation}\label{alpha_def}
        m(x_0,x) = \max\left( (\log(x_0))^{(2B-3)/2}, (\log(x))^{(2B-3)/2}\right).
    \end{equation}
\end{lemma}
\begin{proof}
Since $\varepsilon_{\theta,asymp}(t)$ provides an admissible bound on $\theta(t)$ for all $t\geq x_0$, we have
\begin{equation}\label{int_step1}
    \int_{x_0}^{x} \left| \frac{\theta(t) - t}{t(\log(t))^2} \right| dt  \leq \int_{x_0}^{x} \frac{\varepsilon_{\theta,asymp}(t)}{(\log(t))^2} = \frac{ A_{\theta}}{R^B} \int_{x_0}^{x} (\log(t))^{B-2}\exp\left(-C\sqrt{\frac{\log(t)}{R}}\right) dt.
\end{equation}
We perform the substitution $u = \sqrt{\log(t)}$ 
and note that $u^{2B-3} \leq m(x_0,x)$ as defined in \eqref{alpha_def}.
Thus \eqref{int_step1} is bounded above by
\begin{equation}\label{int_step3}
    \frac{2 A_{\theta} m(x_0,x)}{R^B} \int_{\sqrt{\log(x_0)}}^{\sqrt{\log(x)}} \exp\left( u^2 - \frac{Cu}{\sqrt{R}} \right)du.
\end{equation}
Then, by completing the square $u^2 - \frac{Cu}{\sqrt{R}} = (u - \frac{C}{2\sqrt{R}})^2 - \frac{C^2}{4R}$
and 
doing the substitution $v = u - \frac{C}{2\sqrt{R}}$, \eqref{int_step3} becomes 
\begin{equation}\label{int_step5}
\frac{2 A_{\theta} m(x_0,x)}{R^B} \exp\left(-\frac{C^2}{4R}\right) \int_{\sqrt{\log(x_0)}-\frac{C}{2\sqrt{R}}}^{\sqrt{\log(x)} - \frac{C}{2\sqrt{R}}} \exp\left( v^2 \right)dv .
\end{equation}
Now we have
\begin{equation}\label{int_step6}
\begin{split}
 \int_{\sqrt{\log(x_0)}-\frac{C}{2\sqrt{R}}}^{\sqrt{\log(x)} - \frac{C}{2\sqrt{R}}} \exp\left( v^2 \right)dv 
& \leq  \int_{0}^{\sqrt{\log(x)}  - \frac{C}{2\sqrt{R}}} \exp\left( v^2 \right)dv
\\& =  x \exp\left(\frac{C^2}{4R}\right)\exp\left(   - C\sqrt{\frac{\log(x)}{R}}  \right)   D_{+}\left(\sqrt{\log(x)}  - \frac{C}{2\sqrt{R}}\right).
   \end{split}
\end{equation}
Combining \eqref{int_step5} with \eqref{int_step6} completes the proof.
\end{proof}

\subsection{From an asymptotic bound for $\psi(x)$ to one for $\theta(x)$}
It is well known that given an admissible bound for $\psi$, one may obtain one for $\theta$.
This is made explicit by \cite[Corollary 14.1]{BKLNW21}:
\begin{prop}\label{epsilon_asymp_theta_prop}
Suppose $ A_{\psi},B,C,x_0$ give an admissible asymptotic bound \eqref{bnd-asymp-E-eps} for $E_{\psi}(x)$. If $B > \frac{C^2}{8R}$, then $ A_{\theta},B,C,x_0$ give an admissible asymptotic bound \eqref{bnd-asymp-E-eps} for $E_{\theta}(x)$,  for all $x \geq x_0$ where
  \begin{align}
& \label{def-Atheta}
 A_{\theta} =  A_{\psi}(1 + \nu_{asymp}(x_0)),
 \end{align}
   \begin{align}
& \label{nu_asymp}
\nu_{asymp}(x_0) = \frac{1}{ A_{\psi}}\left( \frac{R}{\log(x_0)} \right)^B\exp\left(C\sqrt{\frac{\log(x_0)}{R}} \right) \left( a_1\log(x_0)x_0^{-1/2} + a_2\log(x_0)x_0^{-2/3} \right),
    \end{align}
    and $a_1,a_2$ are defined in \cite[Corollary 5.1]{BKLNW21} and depend on $x_0$.
\end{prop}
\begin{corollary}\label{rem-always-theta}
Let $R=5.5666305$.
For all $x\ge 2$, 
\begin{equation}\label{def-epsilon-theta}
E_{\theta}(x) \le \varepsilon_{\theta,asymp}(x) = 121.0961\left( \frac{\log(x)}{R} \right)^{3/2} \exp\left( -2 \sqrt{\frac{\log(x)}{R}} \right).
\end{equation}
\end{corollary}
\begin{proof}
By \cite[Corollary 1.3]{FKS}, with $R=5.5666305$, and using the admissible asymptotic bound \eqref{bnd-asymp-E-eps} for $E_{\psi}(x)$ with 
\[ A_{\psi}= 121.096,\ B = 3/2,\ C = 2,\ \text{for all}\ x\ge x_0 = e^{30},\]
we can obtain
\[
\nu_{asymp}(x_0) \le 6.3376 \cdot 10^{-7},\] from which one can conclude an admissible asymptotic bound \eqref{bnd-asymp-E-eps} for $E_{\theta}(x)$ with 
\[ A_{\theta}= 121.0961,\ B = 3/2,\ C = 2,\ \text{for all}\ x\ge x_0 = e^{30}.\]
Additionally, the minimum value of $\varepsilon_{\theta,asymp}(x)$ for $2\leq x\leq e^{30}$ is roughly $2.6271\ldots $ at $x=2$. The results found in \cite[Table~13 and 14]{BKLNW21} give
 \begin{equation*}
 E_{\theta}(x) \leq 1 <   \varepsilon_{\theta,asymp}(2) \leq   \varepsilon_{\theta,asymp}(x)  \ \text{for all}\ 2 \leq x \leq e^{30}.
\qedhere
 \end{equation*}
\end{proof}
\begin{remark}\label{rem-always-theta1}
We also note that for $x_0>e^{1\,000}$,  $B = 3/2$, $C = 2$, and $ A_{\psi}$ being any value from \cite[Table 6]{FKS}, we have
\[\nu_{asymp}(x_0) \le 10^{-200}.\] 
Thus, one easily verifies that the rounding up involved in forming  \cite[Table 6]{FKS} exceeds the rounding up 
also needed to apply this step. 
Consequently we may use values from $ A_{\theta}$ taken from \cite[Table 6]{FKS} directly but this does, in contrast to Corollary \ref{rem-always-theta}, require the assumption $x>x_0$, as per that table.
\end{remark}
 
 \begin{proof}[Proof of Proposition \ref{epsilon_asymp_theta_prop}]
The proof of \cite[Corollary~14.1]{BKLNW21} essentially proves the proposition, but requires that $x_0 \geq e^{1\,000}$ to conclude that the function
\begin{equation}\label{theta_asymp_1}
    1 + \frac{a_1\exp(C\sqrt{\frac{\log(x)}{R}})}{  A_{\psi}\sqrt{x}\left(\frac{\log(x)}{R}\right)^B} + \frac{a_2\exp(C\sqrt{\frac{\log(x)}{R}})}{ A_{\psi}x^{\frac{2}{3}}\left(\frac{\log(x)}{R}\right)^B}
    =
1 + \frac{a_1}{ A_{\psi}} g\left(\frac{1}{2}, -B, \frac{C}{\sqrt{R}},x\right) + \frac{a_2}{ A_{\psi}} g\left(\frac{2}{3}, -B, \frac{C}{\sqrt{R}},x\right)
\end{equation} 
is decreasing.
By Lemma \ref{decreasing_func}, since $B > \frac{C^2}{8R}$, the function is actually decreasing for all $x$.
\end{proof}

We are now in a position to prove our bound on $E_{\pi}(x)$ given in Theorem \ref{prop_asym_pi}.

\subsection{From an asymptotic bound for $\theta(x)$ to one for $\pi(x)$}
\begin{proof}[Proof of Theorem \ref{prop_asym_pi}]
We assume that $\left(\pi(x_0) - \Li(x_0)\right)$ can be numerically calculated. Thus we use \eqref{Stieltjes}  to rewrite $ \left(\pi(x) - \Li(x) \right) - \left(\pi(x_0) - \Li(x_0) \right)$, so that
\begin{equation}\label{starting_bound_on_pi}
\begin{split}
    |\pi(x) - \Li(x)| 
    & = \left| \frac{\theta(x)-x}{\log(x)} - \frac{\theta(x_0)-x_0}{\log(x_0)} + \int_{x_0}^{x} \frac{\theta(t)-t}{t(\log(t))^2}dt + \pi(x_0) - \Li(x_0) \right| \\
    & \leq \left| \pi(x_0) - \Li(x_0) - \frac{\theta(x_0) - x_0}{\log(x_0)} \right| + \left| \frac{\theta(x)-x}{\log(x)} \right| + \left| \int_{x_0}^{x} \frac{\theta(t)-t}{t(\log(t))^2}dt \right|.
\end{split}
\end{equation}
We use the assumption ($\varepsilon_{\theta,asymp}(x)$ provides an admissible bound on $\theta(x)$ for all $x\geq  x_0$)
to bound $    \left| \frac{\theta(x) - x}{\log(x)} \right|$
and Lemma \ref{integral_asymp_bound} to bound $    \left| \int_{x_0}^{x} \frac{\theta(t)-t}{t(\log(t))^2}dt \right|$.
We obtain
\begin{equation}\label{eqn-pi-minus-Li}
\begin{split}
    |\pi(x) - \Li(x)| 
    & \leq \left| \pi(x_0) - \Li(x_0) - \frac{\theta(x_0) - x_0}{\log(x_0)} \right| 
    + \frac{x\varepsilon_{\theta,asymp}(x)}{\log(x)}
    \\
    & \quad + \frac{2 A_{\theta}}{R^B}xm(x_0,x)\exp\left( - C\sqrt{\frac{\log(x)}{R}} \right)D_{+}\left( \sqrt{\log(x)} - \frac{C}{2\sqrt{R}} \right).
\end{split}
\end{equation}
We recall that $x\geq x_1\geq x_0$.
Note that, by Corollary \ref{decreasing_func_corollary}, 
$ \frac{\log(x)}{x\varepsilon_{\theta,asymp}(x)} = \frac{1}{ A_{\theta}}g\left(1, 1-B, \frac{C}{\sqrt{R}}, x\right)$
is decreasing for all $x$. Thus, 
\begin{equation}\label{main_proof_step_0}
\frac{\log(x)}{x\varepsilon_{\theta,asymp}(x)} 
\leq \frac{\log(x_1)}{x_1\varepsilon_{\theta,asymp}(x_1)} .
\end{equation}
In addition, we have the simplification 
\begin{equation}\label{main_proof_step_1}
\frac{\log(x)}{x\varepsilon_{\theta,asymp}(x)} \frac{2 A_{\theta}}{R^B}xm(x_0,x) e^{ - C\sqrt{\frac{\log(x)}{R}} }
= 2m(x_0,x)(\log(x))^{1-B} 
= \frac{2}{\sqrt{ \log(x)}} 
\le \frac{2}{\sqrt{ \log(x_1)}} ,
\end{equation}
by definition \eqref{bnd-asymp-E-eps} of $\varepsilon_{\theta,asymp}(x)$ and by $m(x_0,x) = (\log(x))^{(2B-3)/2}$, since $B \geq 3/2$.
Finally, since $\sqrt{\log(x_1)} - \frac{C}{2\sqrt{R}} > 1$, the Dawson function decreases for all $x\ge x_1$:
\begin{equation}\label{bnd_Daw}
D_{+}\left( \sqrt{\log(x)} - \frac{C}{2\sqrt{R}} \right)
\leq D_{+}\left( \sqrt{\log(x_1)} - \frac{C}{2\sqrt{R}} \right).
\end{equation}
We conclude by combining \eqref{eqn-pi-minus-Li}, \eqref{main_proof_step_0}, \eqref{main_proof_step_1}, and \eqref{bnd_Daw}:
\[
\frac{\left| \pi(x) - \Li(x) \right| }{ \frac{x\varepsilon_{\theta,asymp}(x)}{\log(x)} }
\le \frac{ \log(x_1)}{x_1 \varepsilon_{\theta,asymp}(x_1)}  \left| \pi(x_0) - \Li(x_0) - \frac{\theta(x_0) - x_0}{\log(x_0)} \right|  
+1 + \frac{2 D_{+}\big( \sqrt{\log(x_1)} - \frac{C}{2\sqrt{R}} \big)}{\sqrt{\log(x_1)}},
\]
from which we deduce the announced bound.
\end{proof}
\begin{remark}\label{rem-pi2theta}
 One can still obtain a formula for admissible asymptotic bounds on $\pi$ while removing the restriction imposed on $B$.
 For the purposes of clarity in both the statement and proof of Theorem \ref{prop_asym_pi}, we have chosen to limit our study to only those admissible asymptotic bounds on $\theta$ with a satisfactory $B$.
 The main difference is that in the proof, when considering \eqref{main_proof_step_1}, one will use a different bound for $m(x_0,x)$, and hence obtain a different final formula.
\end{remark}

\section{numerical Bounds for $\pi(x)$}\label{section:numeric_bounds}

The aim of this section is to convert between numerical bounds, as  in Definition \ref{epsilon_medium_theta_def}. 
The main difference between our asymptotic result and our numerical result is in how we estimate the integral in \eqref{important_integral}. 

\subsection{From a numerical bound for $\psi(x)$ to one for $\theta(x)$}
As with asymptotic bounds, one can convert numerical bounds on $\psi$ to bounds on $\theta$, and then on $\pi$. 
The following is both an effective and simple method.
\begin{prop}\label{prop:numpsitotheta}
Let $x>x_0> 2$. If 
$ E_{\psi}(x) \leq \varepsilon_{\psi,num}(x_0) $,
then 
\[  -\varepsilon_{\theta,num}(x_0)  \leq \frac{\theta(x)-x}{x}\leq \varepsilon_{\psi,num}(x_0) < \varepsilon_{\theta,num}(x_0)  ,\]
where 
\begin{equation*}
\varepsilon_{\theta,num}(x_0) = \varepsilon_{\psi,num}(x_0) + 1.00000002(x_0^{-1/2} + x_0^{-2/3} +  x_0^{-4/5}) + 0.94(x_0^{-3/4}  + x_0^{-5/6} +x_0^{-9/10}) .
\end{equation*}
\end{prop}
\begin{proof}
It is obvious that $\frac{\theta(x)-x}{x} \leq \frac{\psi(x)-x}{x}$
so we focus on the lower bound.

We have that
\[ \frac{\theta(x)-x}{x} = \frac{\psi(x)-x}{x} + \frac{\theta(x)-\psi(x)}{x} \]
by \cite[Theorem 1]{Pereira}
\[  \psi(x)  - \theta(x) \leq \psi(x^{1/2}) + \psi(x^{1/3})  + \psi(x^{1/5}).\]
Now in several intervals we use \cite[Theorem 2]{Bu18}, that for $0<x<11$, $\psi(x)<x$, and that $\varepsilon_{\psi,num}(10^{19})<2\cdot10^{-8}$. In particular  when $2 < x < 10^{38}$ 
\[   \psi(x^{1/2}) + \psi(x^{1/3})  + \psi(x^{1/5}) \leq  x^{1/2} + x^{1/3} +x^{1/5}+ 0.94(x^{1/4} + x^{1/6} + x^{1/10}) \]
when $10^{38} \leq x < 10^{54}$ 
\[   \psi(x^{1/2}) + \psi(x^{1/3})  + \psi(x^{1/5}) \leq  1.00000002x^{1/2} + x^{1/3} +x^{1/5}+ 0.94(x^{1/6} +x^{1/10}) \]
when $10^{54} \leq x < 10^{95}$
\[   \psi(x^{1/2}) + \psi(x^{1/3})  + \psi(x^{1/5}) \leq   1.00000002(x^{1/2} + x^{1/3}) +  x^{1/5} +0.94 x^{1/10} \]
and finally when $x \ge 10^{95}$ 
\[   \psi(x^{1/2}) + \psi(x^{1/3})  + \psi(x^{1/5}) \leq   1.00000002(x^{1/2} + x^{1/3} +  x^{1/5}).  \]
The result follows by combining the worst coefficients from all cases and dividing by $x$.
\end{proof}

\begin{remark}\label{rem:numerictheta}
Values for $\varepsilon_{\theta,num}(x_1)$, as a step function, can be obtained from \cite[Tables 13, 14]{BKLNW21}.
In Tables \ref{table:theta-small}, \ref{table:theta-med} and \ref{table:theta} we will update the values there as follows:
\begin{itemize}
\item for small values of $x_0$, those with $x_0\leq  10^{19}$, we use the maximum of \cite[Theorem 2]{Bu18} and the value from $x_0= 10^{19}$ as described treating $x_0$ as an intermediate value.
\item for intermediate values of $x_0$, those with $ 10^{19} \leq  x_0 < e^{2\,073}$, we apply Proposition \ref{prop:numpsitotheta} to bounds for $\psi$ which are computed using \cite[Theorem~1]{Bu16} as described in \cite[Theorem 16]{BKLNW21}.
 These are recomputed using the new RH verification of $3\cdot 10^{12}$ from \cite{PlaTru21RH}.
\item for larger values of $x_0$, those with $ x_0\ge e^{2\,073}$, we apply Proposition \ref{prop:numpsitotheta} to find bounds for $\psi$ which are computed as described in \cite[Theorem 1.1]{FKS}.
\end{itemize}
\end{remark}

\subsection{Preliminary lemmas}

We will need the following lemmas for the proof of Theorem \ref{prop_num_pi}.

\begin{lemma}\label{sum_over_bi}
    Assume $x_0$ and $x_1$ are real numbers such that $x_1 > x_0 \geq 2$. Let $N \in \N$ and let $\{ b_i \}_{i=1}^{N}$ be a finite partition of the interval $[x_0,x_1]$. Then,
    \begin{equation}\label{sum_over_bi_result}
        \left| \int_{x_0}^{x_1} \frac{\theta(t) - t}{t(\log(t))^2}dt \right| \leq \sum_{i=1}^{N-1} \varepsilon_{\theta,num}(e^{b_i})\left( \Li(e^{b_{i+1}}) - \Li(e^{b_i}) + \frac{e^{b_i}}{b_i} - \frac{e^{b_{i+1}}}{b_{i+1}} \right).
    \end{equation} 
\end{lemma}
\begin{proof}
We split the integral in \eqref{sum_over_bi_result} at each $b_i$ and apply 
the bound 
\[
\left| \frac{\theta(t) - t}{t } \right| \le \varepsilon_{\theta,num}(e^{b_i}), \quad \text{for every }\ e^{b_i}\le t< e^{b_i+1}.
\]
Thus 
\[
        \left| \int_{x_0}^{x_1} \frac{\theta(t) - t}{t(\log(t))^2}dt \right|
   \leq    \sum_{i=1}^{N-1} \int_{e^{b_i}}^{e^{b_{i+1}}} \left| \frac{\theta(t)-t}{t(\log(t))^2} \right|dt 
   \leq \sum_{i=1}^{N-1} \varepsilon_{\theta,num}(e^{b_i}) \int_{e^{b_i}}^{e^{b_{i+1}}} \frac{dt}{(\log(t))^2} .
  \]
We conclude by using the identity: for all $2 \leq a < b$,
\begin{equation}\label{int_1OverLogSquared}
    \int_{a}^{b} \frac{dt}{(\log(t))^2} = \Li(b) - \frac{b}{\log(b)} - \left( \Li(a) - \frac{a}{\log(a)} \right) .
    \qedhere
\end{equation}
\end{proof}
\begin{lemma}\label{max_of_logxLixOverx}
Assume $x \ge 6.58$. Then we have $\Li(x)  -  \frac{x}{\log(x)}$ is strictly increasing and 
\[ \Li(x) - \frac{x}{\log(x)} > \frac{x-6.58}{(\log x)^2}>0 .\]
\end{lemma}
\begin{proof}
Differentiate 
$\displaystyle \frac{d}{dx} \left(  \Li(x) - \frac{x}{\log(x)} \right) = \frac{1}{\log(x)} +  \frac{1 - \log(x)}{(\log(x))^2} = \frac{1}{(\log(x))^2} $
to see that the difference is strictly increasing. Evaluating at $x=6.58$ and applying the mean value theorem gives the announced result.
\end{proof}
For other similar bounds on the logarithmic function, see \cite[Lemma 2.3]{DeSaTr15}
\subsection{From a numerical bound for $\theta(x)$ to one for $\pi(x)$}
\begin{proof}[Proof of Theorem \ref{prop_num_pi}]
We recall inequality \eqref{starting_bound_on_pi} 
\begin{equation} 
    |\pi(x) - \Li(x)| 
 \leq \left| \frac{\theta(x)-x}{\log(x)} \right| + \left| \pi(x_0) - \Li(x_0) - \frac{\theta(x_0) - x_0}{\log(x_0)} \right| + \left| \int_{x_0}^{x} \frac{\theta(t)-t}{t(\log(t))^2}dt \right| .
\end{equation}
By the assumption that $\varepsilon_{\theta,num}(b_i)$ is an admissible bound for $\theta(x)$, for all $i=1,\ldots,N$, we have
\begin{equation}\label{step1}
    \left| \frac{\theta(x) - x}{\log(x)} \right| \leq \varepsilon_{\theta,num}(x_1)\, \frac{x}{\log(x)} \ \text{for all}\ x\ge x_1.
\end{equation}
 Now by Lemma \ref{sum_over_bi}, we have
 \begin{align}
& \label{step21} 
\left| \int_{x_0}^{x_1} \frac{\theta(t)-t}{t(\log(t))^2}dt \right| 
 \leq \sum_{i=1}^{N-1} \varepsilon_{\theta,num}(e^{b_i})\left( \Li(e^{b_{i+1}}) - \Li(e^{b_i}) + \frac{e^{b_i}}{b_i} - \frac{e^{b_{i+1}}}{b_{i+1}} \right) , and \\
& \label{step22}
\left| \int_{x_1}^{x} \frac{\theta(t)-t}{t(\log(t))^2}dt \right| \le  \varepsilon_{\theta,num}(x_1)\int_{x_1}^{x} \frac{1}{(\log(t))^2}dt .
 \end{align}
Thus, inequality \eqref{starting_bound_on_pi} becomes
\begin{align*}
    \frac{|\pi(x) - \Li(x)|}{\frac{x}{\log(x)}}
     \leq& \varepsilon_{\theta,num}(x_1) +  \frac{\log(x)}{x}\left| \pi(x_0) - \Li(x_0) - \frac{\theta(x_0) - x_0}{\log(x_0)} \right| \\
    &  + \frac{\log(x)}{x} \sum_{i=1}^{N-1} \varepsilon_{\theta,num}(e^{b_i})\left( \Li(e^{b_{i+1}}) - \Li(e^{b_i}) + \frac{e^{b_i}}{b_i} - \frac{e^{b_{i+1}}}{b_{i+1}} \right) \\
    &+   \varepsilon_{\theta,num}(x_1)\frac{\log(x)}{x}\int_{x_1}^{x} \frac{1}{(\log(t))^2}dt .
\end{align*}

We now consider the term 
\begin{equation}   \label{term3} 
f(x)= \frac{\log(x)}{x}\int_{x_1}^{x} \frac{1}{(\log(t))^2}dt 
=\frac{\log(x)}{x} \left(  \Li(x) - \frac{x}{\log(x)}  - \Li(x_1) +\frac{x_1}{\log(x_1)} \right).
 \end{equation}
Using integration by part, its derivative can be written as 
\[
  f'(x) 
   =     -\frac{1}{x(\log(x))^2} +  \frac{2}{x(\log(x))^3} + \frac{\log(x)-1}{x^2}\left(  \frac{x_1}{(\log(x_1))^2}  + \frac{2x_1}{(\log(x_1))^3} -  \int_{x_1}^{x} \frac{6}{(\log(t))^4}dt \right) .
  \]
From which we see that $f'(x_1) = \frac{1}{\log(x_1)}> 0$, and that $f'(x)$ is eventually negative.
Thus there exists a critical point for $f(x)$ to the right of $x_1$.
Moreover, by bounding $\int_{x_1}^{x} \frac{6}{\log^4(t)}dt < 6 \frac{x-x_1}{\log(x_1)^4}$, one finds that  $f'(x_1\log(x_1))>0$ if $x_1>e$.


Now we write $f'(x)=\frac{f_1(x)}{x^2}$ with
\[ f_1(x)=   \frac{x}{\log(x)} - (\log(x)-1)\int_{x_1}^{x} \frac{1}{(\log(t))^2}dt .\] 
Its derivative is
$f_1'(x)= - \frac{1}{x}\int_{x_1}^{x} \frac{1}{(\log(t))^2} $, 
which is negative for $x>x_1$. Thus $f_1(x)$ decreases and vanishes at most once, giving $f(x)$ at most one critical point, $x_m>x_1$, which is then the maximum of $f(x)$. In other words, $x_m$ satisfies
$f_1(x_m)=0$, i.e. 
$  \Li(x_m)-\Li(x_1) + \frac{x_1}{\log(x_1)} = -\frac{x_m}{ 1-\log(x_m)},
$ 
which shows that $f(x)$ attains its maximum at $x=x_m$, where
\[f(x_m)= \frac{\log(x_m)}{x_m} \left( - \frac{x_m}{\log(x_m)} -\frac{x_m}{ 1-\log(x_m)} \right)
= \frac{1}{ \log(x_m)-1 }.
\]
Now, because $x_m> x_1\log(x_1)$ we obtain the bound 
\[ f(x) < \frac{1}{\log(x_1) + \log(\log(x_1))-1} ,\] 
which gives \eqref{mu_num_def}.

To obtain \eqref{mu_num_defx1x2}, notice that by assumption $x_1 \leq x \leq x_2 \leq x_1\log(x_1) < x_m$, so that 
\[ f(x) \leq f(x_2)  = \frac{\log(x_2)}{x_2} \left(  \Li(x_2) - \frac{x_2}{\log(x_2)}  - \Li(x_1) +\frac{x_1}{\log(x_1)} \right). \qedhere \]
\end{proof}

\section{Corollaries}\label{Section-Cor}

Together with Proposition \ref{epsilon_asymp_theta_prop}, we deduce a direct asymptotic bound from $\psi$ to $\pi$:
\begin{corollary}\label{asymp-bnd-psi2pi}
Let $B \geq \max(\frac{3}{2}, 1+\frac{C^2}{16R})$. 
Let $x_0,x_1>0$ such that $x_1$ satisfies \eqref{def-x1}.
If $ A_{\psi},B,C,x_0$ give an admissible asymptotic bound \eqref{bnd-asymp-E-eps} for $E_{\psi}(x)$, then $ A_{\pi},B,C,x_1$ give an admissible bound for $E_{\pi}(x)$ with
\[ A_{\pi} = (1 + \nu_{asymp}(x_0))(1 + \mu_{asymp}(x_0,x_1)) A_{\psi},\]
where 
$\nu_{asymp}(x_0)$ is defined in \eqref{nu_asymp}, $\mu_{asymp}(x_0,x_1)$ is defined in \eqref{mu_asymp_def}.
In other words, 
\begin{equation}
\label{epsilon_pi_asymp_def2}
 E_{\pi}(x)\leq \varepsilon_{\pi,asymp}(x) = (1 + \nu_{asymp}(x_0))(1 + \mu_{asymp}(x_0,x_1))\varepsilon_{\psi,asymp}(x),\ \text{for all}\ x\geq x_1.
\end{equation}
\end{corollary}
This allows us to directly deduce asymptotic bounds for $\pi(x)$ from \cite[Table 6]{FKS}.



Using Proposition \ref{epsilon_asymp_theta_prop}, Theorems \ref{prop_asym_pi} and \ref{prop_num_pi} with the asymptotic bounds for $\psi$ coming from \cite{FKS}, and the numerical bounds for 
$\theta$ computed as per Remark \ref{rem:numerictheta} and given in Table \ref{tab:pi}, we obtain the following:
\begin{corollary}\label{epsilon_asymp_pi_explicit}
For all $x\geq 2$ we have
\begin{equation}\label{large_x_result}
\left| \pi(x) - \Li(x) \right| \leq 9.2211\, x\sqrt{\log(x)} \exp\left( -0.84768363 \sqrt{\log(x)} \right).
\end{equation}
\end{corollary}
\begin{proof}
We fixed $R=5.5666305$, $x_0 = 40.787732519...$, and $x_1 = e^{20\,000}$
and use \eqref{bnd-asymp-E-eps} with values $ A_{\theta} = 121.0961$, $B = 3/2$ and $C = 2$. By Corollary \ref{rem-always-theta}, these are admissible for all $x\geq 2$, so we can apply Theorem \ref{prop_asym_pi} and calculate that 
 \begin{equation}\label{eg-mu} \mu_{asymp}(40.78\ldots,e^{20\,000}) \leq 5.01516\cdot10^{-5}. \end{equation}
 This implies that $ A_{\pi}=121.103$ is admissible for all $x \geq e^{20\,000}$.

As in the proof of \cite[Lemmas 5.2 and 5.3]{FKS} one may verify that the numerical results obtainable from Theorem \ref{prop_num_pi}, using Corollary \ref{cor_num_pi}, may be interpolated as a step function
to give a bound on $E_\pi(x)$ of the shape $\varepsilon_{\pi,asymp}(x)$.
In this way we obtain that  $ A_{\pi} = 121.107$ is admissible for  $x>2$,  
Note that the subdivisions we use are essentially the same as used in \cite[Lemmas 5.2 and 5.3]{FKS}.
In Table \ref{interp-table} we give a sampling of the relevant values, more of the values of $ \varepsilon_{\pi,num}(x_1)$  can be found in Table \ref{tab:pi}. 
Far more detailed versions of these tables will be posted online in \cite{FKS3}.
\end{proof}
We point out several results which may be useful when bounds that are tighter than  \eqref{large_x_result} are needed for small $x$.
%
\begin{corollary}\label{cor:alt1}
$ A_{\pi}$, $B$, $C$, and $x_0$ as in Table \ref{table:alt1} provide an asymptotic bound \eqref{bnd-asymp-E-eps} for $E_{\pi}$.
 %
\end{corollary}
\begin{corollary}\label{cor:alt2}
We have the following bounds $E_{\pi}(x)\le \mathcal{B}(x)$, where $ \mathcal{B}(x)$ is given in Table \ref{table:alt2}.
\end{corollary}
\begin{remark}
The above refines \cite[Theorem 2]{Tru16}, \cite[Lemma 6]{PlaTru21ET}, and \cite[Proposition 3.1]{Joh22}.
\end{remark}
\begin{proof}[Proof of Corollaries \ref{cor:alt1} and \ref{cor:alt2}]
The bounds of the form $\varepsilon_{\pi,asymp}(x)$ come from selecting a value $A$ for which  \eqref{large_x_result}  provides a better bound at $x = e^{7\,500}$ and from verifying that  \eqref{large_x_result}  decreases faster beyond this point.
This final verification proceeds by looking at the derivative of the ratio as in Lemma \ref{decreasing_func}. To verify these still hold for smaller $x$, we proceed as below.
To verify the results for any $x$ in $\log(10^{19}) < \log(x) < 100\,000$, one simply proceeds as in \cite[Lemmas 5.2, 5.3]{FKS} and interpolates the numerical results of Theorem \ref{prop_num_pi}. For instance, we use the values in Table \ref{tab:pi} as a step function 
and verifies that it provides a tighter bound than we are claiming. Note that our verification uses a more refined collection of values than those provided in Table \ref{tab:pi} or the tables posted online in \cite{FKS3}.
To verify results for $x<10^{19}$, one compares against the results from \cite[Theorem 2]{Bu18}, or one checks directly for particularly small $x$.
\end{proof}\ \newline
Next we point out the weaker, but potentially easier to apply result:
\begin{corollary}\label{cor:weak}
For all $x \geq 2$, 
\begin{equation}\label{explicit_theroem_result}
\left| \pi(x) - \Li(x) \right| \leq 0.4298\frac{x}{\log(x)}.
\end{equation} 
\end{corollary}
\begin{proof}
We numerically verify that the inequality in \eqref{explicit_theroem_result} holds by showing that, for $1 \leq n \leq 25$ and all $x \in [p_n, p_{n+1}]$,
\[ \left| \frac{\log(x)}{x} \left(\pi(x) - \Li(x)\right) \right| \leq \left| \frac{\log(p_n)}{p_n} \left(\pi(p_n) - \Li(p_{n+1}\right)) \right| \leq 0.4298. \]
For $x$ satisfying $p_{25} = 97 \leq x \leq 10^{19}$, we  use \cite[Theorem~2]{Bu18}
and verify
\[ \mathcal{E}(x) = \frac{1}{\sqrt{x}} \left( 1.95 + \frac{3.9}{\log(x)} + \frac{19.5}{(\log(x))^2} \right) \leq 0.4298. \]
For $x> 10^{19}$, we use Theorem \ref{prop_num_pi} as well as values for $\varepsilon_{\pi,num}(x)$ found in Table \ref{tab:pi} to conclude 
\[ \varepsilon_{\pi,num}(x) \leq 0.4298. \qedhere\]
\end{proof}

\section{Tables}
\begin{small}

\begin{table}[h!]
\caption{Values for $\varepsilon_{\psi,num}(x_0)$ and $\varepsilon_{\theta,num}(x_0)$ as defined in \eqref{bnd-num-E-eps}, as computed in \cite[Theorem 2]{Bu18} and using \cite[Theorem 1]{Bu16} to extend to values $x > 10^{19}$.~Note that for  $36<\log(x)<\log(10^{19})$ the bound from \cite[Theorem 2]{Bu18} is tighter than what we provide but is not known to hold for larger $x$.~}\label{table:theta-small}

\resizebox{\textwidth}{!}
{
\begin{tabular}{|c|c|c|}
\hline
$\log(x_1)$ & $\varepsilon_{\psi,num}(x_0)$ & $\varepsilon_{\theta,num}(x_0)$\\
\hline
$4$ & 0.12722 & 0.27880 \\
$5$ & 0.077160 & 0.16910 \\
$6$ & 0.046800 & 0.10257 \\
$7$ & 0.028386 & 0.058885 \\
$8$ & 0.017217 & 0.035716 \\
$9$ & 0.010443 & 0.021663 \\
$10$ & 0.0063337 & 0.013139 \\
$11$ & 0.0038416 & 0.0079693 \\
$12$ & 0.0023301 & 0.0048336 \\
$13$ & 0.0014133 & 0.0029318 \\
$14$ & 0.00085717 & 0.0017782 \\
\hline
\end{tabular}
\hfil
\begin{tabular}{|c|c|c|}
\hline
$\log(x_1)$ & $\varepsilon_{\psi,num}(x_0)$& $\varepsilon_{\theta,num}(x_0)$\\
\hline
$15$ & 0.00051990 & 0.0010786 \\
$16$ & 0.00031534 & 0.00065416 \\
$17$ & 0.00019127 & 0.00039677 \\
$18$ & 0.00011601 & 0.00024065 \\
$19$ & 7.0361 e-5 & 0.00014597 \\
$20$ & 4.2676 e-5 & 8.8530 e-5 \\
$21$ & 2.5885 e-5 & 5.3697 e-5 \\
$22$ & 1.5700 e-5 & 3.2569 e-5 \\
$23$ & 9.5223 e-6 & 1.9754 e-5 \\
$24$ & 5.7756 e-6 & 1.1982 e-5 \\
$25$ & 3.5031 e-6 & 7.2670 e-6 \\
\hline
\end{tabular}
\hfil
\begin{tabular}{|c|c|c|}
\hline
$\log(x_1)$ & $\varepsilon_{\psi,num}(x_0)$ &$\varepsilon_{\theta,num}(x_0)$\\
\hline
$26$ & 2.1248 e-6 & 4.4077 e-6 \\
$27$ & 1.2888 e-6 & 2.6734 e-6 \\
$28$ & 7.8164 e-7 & 1.6215 e-6 \\
$29$ & 4.7409 e-7 & 9.8348 e-7 \\
$30$ & 2.8755 e-7 & 5.9651 e-7 \\
$31$ & 1.7441 e-7 & 3.6181 e-7 \\
$32$ & 1.0579 e-7 & 2.1945 e-7 \\
$33$ & 6.4161 e-8 & 1.3310 e-7 \\
$34$ & 3.8916 e-8 & 8.0729 e-8 \\
$35$ & 2.3604 e-8 & 4.8965 e-8 \\
$36$ & 2.3604 e-8 & 2.9699 e-8 \\
\hline
\end{tabular}
}
\end{table}

\begin{table}[h!]
\caption{Values for $\alpha, c, T$, and $\varepsilon_{\psi,num}(x_1)$ computed as in \cite[Theorem 1]{Bu16}, and for $\varepsilon_{\theta,num}(x_0)$ as in Proposition \ref{prop:numpsitotheta}.~
The first line also uses \cite[Theorem 2]{Bu18}.~Note that the $\alpha$, $c$ and $T$ values are rounded.}\label{table:theta-med}
\centering
{
\begin{tabular}{|c|ccc|c|c|}
\hline
$\log(x_0)$ & $\alpha$ & c & T & $\varepsilon_{\psi,num}(x_0)$ & $\varepsilon_{\theta,num}(x_0)$ \\
\hline
$37$ & &  &   & 1.9220 e-8 & 1.9537 e-9 \\
$\log(10^{19})$ & 0.1211 & 23.93 &  3.045 e9 & 1.9220 e-8 & 1.9537 e-9 \\
$44$ & 0.1049 & 37.88 & 4.567 e9 & 1.7144 e-8 & 1.7423 e-8 \\
$45$ & 0.1122 & 33.03 & 6.851 e9 & 1.0906 e-8 & 1.1075 e-8 \\
$46$ & 0.1163 & 30.58 & 1.028 e10 & 6.9310 e-9 & 7.0337 e-9 \\
$47$ & 0.1170 & 30.01 & 1.541 e10 & 4.4023 e-9 & 4.4645 e-9 \\
$48$ & 0.1168 & 29.92 & 2.312 e10 & 2.7948 e-9 & 2.8326 e-9 \\
$49$ & 0.1161 & 30.02 & 3.468 e10 & 1.7737 e-9 & 1.7966 e-9 \\
$50$ & 0.1065 & 36.60 & 7.804 e10 & 1.1197 e-9 & 1.1336 e-9 \\
$51$ & 0.1114 & 33.31 & 1.171 e11 & 7.0834 e-10 & 7.1676 e-10 \\
$52$ & 0.1122 & 32.64 & 1.756 e11 & 4.4790 e-10 & 4.5301 e-10 \\
$53$ & 0.1121 & 32.50 & 2.634 e11 & 2.8313 e-10 & 2.8623 e-10 \\
$54$ & 0.1115 & 32.58 & 3.951 e11 & 1.7893 e-10 & 1.8081 e-10 \\
$55$ & 0.1032 & 38.95 & 8.889 e11 & 1.1254 e-10 & 1.1368 e-10 \\
$56$ & 0.1075 & 35.80 & 1.333 e12 & 7.0920 e-11 & 7.1612 e-11 \\
$57$ & 0.1082 & 35.13 & 2.000 e12 & 4.4676 e-11 & 4.5096 e-11 \\
$58$ & 0.1080 & 35.01 & 3.000 e12 & 2.8138 e-11 & 2.8393 e-11 \\
$59$ & 0.1076 & 33.92 & 3.000 e12 & 1.8007 e-11 & 1.8161 e-11 \\
$60$ & 0.1061 & 33.57 & 3.000 e12 & 1.1851 e-11 & 1.1945 e-11 \\
$61$ & 0.1044 & 33.41 & 3.000 e12 & 8.1094 e-12 & 8.1662 e-12 \\
$62$ & 0.1026 & 33.33 & 3.000 e12 & 5.8340 e-12 & 5.8684 e-12 \\
$63$ & 0.1008 & 33.28 & 3.000 e12 & 4.4480 e-12 & 4.4689 e-12 \\
$64$ & 0.09906 & 33.26 & 3.000 e12 & 3.6018 e-12 & 3.6145 e-12 \\
$65$ & 0.09736 & 33.24 & 3.000 e12 & 3.0832 e-12 & 3.0909 e-12 \\
$66$ & 0.09571 & 33.24 & 3.000 e12 & 2.7635 e-12 & 2.7682 e-12 \\
$67$ & 0.09411 & 33.24 & 3.000 e12 & 2.5647 e-12 & 2.5675 e-12 \\
$68$ & 0.09257 & 33.24 & 3.000 e12 & 2.4393 e-12 & 2.4410 e-12 \\
$69$ & 0.09107 & 33.24 & 3.000 e12 & 2.3587 e-12 & 2.3597 e-12 \\
$70$ & 0.08962 & 33.25 & 3.000 e12 & 2.3053 e-12 & 2.3059 e-12 \\
$71$ & 0.08822 & 33.25 & 3.000 e12 & 2.2687 e-12 & 2.2690 e-12 \\
$72$ & 0.08687 & 33.26 & 3.000 e12 & 2.2423 e-12 & 2.2425 e-12 \\
$73$ & 0.08555 & 33.26 & 3.000 e12 & 2.2222 e-12 & 2.2224 e-12 \\
$74$ & 0.08428 & 33.26 & 3.000 e12 & 2.2062 e-12 & 2.2063 e-12 \\
$75$ & 0.08305 & 33.27 & 3.000 e12 & 2.1927 e-12 & 2.1928 e-12 \\
\hline
\end{tabular}
}
\end{table}

\begin{table}[h!]
\caption{Values for $\varepsilon_{\theta,num}$ are as calculated in Proposition \ref{prop:numpsitotheta}.
They use values for $\varepsilon_{\psi,num}$ as calculated in \cite[Theorem 1]{Bu16} for $\log(x_0) < 2\,073$ and \cite[Theorem 1.1]{FKS} for $\log(x_0) \ge 2\,073$, respectively.~Note that computed values of $\varepsilon_{\theta,num}$ and $\varepsilon_{\psi,num}$ all round up to the same value.} \label{table:theta}
\centering
\resizebox{0.89\textwidth}{!}
{
\begin{tabular}{|c|c|}
\hline
$\log(x_0)$ & $\varepsilon_{\theta,num}(x_0) $ \\
\hline
 $76$ & 2.1809 e-12 \\
 $77$ & 2.1702 e-12 \\
 $78$ & 2.1602 e-12 \\
 $79$ & 2.1508 e-12 \\
 $80$ & 2.1419 e-12 \\
 $81$ & 2.1333 e-12 \\
 $82$ & 2.1251 e-12 \\
 $83$ & 2.1171 e-12 \\
 $84$ & 2.1093 e-12 \\
 $85$ & 2.1018 e-12 \\
 $86$ & 2.0945 e-12 \\
 $87$ & 2.0874 e-12 \\
 $88$ & 2.0805 e-12 \\
 $89$ & 2.0738 e-12 \\
 $90$ & 2.0672 e-12 \\
 $91$ & 2.0608 e-12 \\
 $92$ & 2.0546 e-12 \\
 $93$ & 2.0485 e-12 \\
 $94$ & 2.0425 e-12 \\
 $95$ & 2.0367 e-12 \\
 $96$ & 2.0311 e-12 \\
 $97$ & 2.0256 e-12 \\
 $98$ & 2.0202 e-12 \\
 $99$ & 2.0149 e-12 \\
 $100$ & 2.0097 e-12 \\
 $110$ & 1.9639 e-12 \\
 $120$ & 1.9264 e-12 \\
 $130$ & 1.8952 e-12 \\
 $140$ & 1.8688 e-12 \\
 $150$ & 1.8461 e-12 \\
 $160$ & 1.8264 e-12 \\
 $170$ & 1.8092 e-12 \\
 $180$ & 1.7940 e-12 \\
 $190$ & 1.7805 e-12 \\
 $200$ & 1.7684 e-12 \\
 $210$ & 1.7575 e-12 \\
 $220$ & 1.7476 e-12 \\
 $230$ & 1.7386 e-12 \\
 $240$ & 1.7304 e-12 \\
 $250$ & 1.7229 e-12 \\
 $260$ & 1.7160 e-12 \\
 $270$ & 1.7095 e-12 \\
 $280$ & 1.7036 e-12 \\
 $290$ & 1.6981 e-12 \\
 $300$ & 1.6930 e-12 \\
 $310$ & 1.6882 e-12 \\
 $320$ & 1.6837 e-12 \\
 $330$ & 1.6795 e-12 \\
 $340$ & 1.6755 e-12 \\
 $350$ & 1.6718 e-12 \\
 $360$ & 1.6683 e-12 \\
\hline
\end{tabular}
\hfil
\begin{tabular}{|c|c|}
\hline
$\log(x_0)$ & $\varepsilon_{\theta,num}(x_0) $ \\
\hline
 $370$ & 1.6650 e-12 \\
 $380$ & 1.6618 e-12 \\
 $390$ & 1.6588 e-12 \\
 $400$ & 1.6560 e-12 \\
 $425$ & 1.6495 e-12 \\
 $450$ & 1.6438 e-12 \\
 $475$ & 1.6387 e-12 \\
 $500$ & 1.6341 e-12 \\
 $525$ & 1.6299 e-12 \\
 $550$ & 1.6261 e-12 \\
 $575$ & 1.6227 e-12 \\
 $600$ & 1.6195 e-12 \\
 $625$ & 1.6166 e-12 \\
 $650$ & 1.6140 e-12 \\
 $675$ & 1.6115 e-12 \\
 $700$ & 1.6092 e-12 \\
 $725$ & 1.6071 e-12 \\
 $750$ & 1.6051 e-12 \\
 $775$ & 1.6032 e-12 \\
 $800$ & 1.6015 e-12 \\
 $825$ & 1.5998 e-12 \\
 $850$ & 1.5983 e-12 \\
 $875$ & 1.5968 e-12 \\
 $900$ & 1.5955 e-12 \\
 $925$ & 1.5942 e-12 \\
 $950$ & 1.5929 e-12 \\
 $975$ & 1.5918 e-12 \\
 $1\,000$ & 1.5907 e-12 \\
 $1\,025$ & 1.5896 e-12 \\
 $1\,050$ & 1.5886 e-12 \\
 $1\,075$ & 1.5877 e-12 \\
 $1\,100$ & 1.5867 e-12 \\
 $1\,125$ & 1.5859 e-12 \\
 $1\,150$ & 1.5850 e-12 \\
 $1\,175$ & 1.5843 e-12 \\
 $1\,200$ & 1.5835 e-12 \\
 $1\,225$ & 1.5828 e-12 \\
 $1\,250$ & 1.5821 e-12 \\
 $1\,275$ & 1.5814 e-12 \\
 $1\,300$ & 1.5807 e-12 \\
 $1\,325$ & 1.5801 e-12 \\
 $1\,350$ & 1.5795 e-12 \\
 $1\,375$ & 1.5789 e-12 \\
 $1\,400$ & 1.5784 e-12 \\
 $1\,425$ & 1.5778 e-12 \\
 $1\,450$ & 1.5773 e-12 \\
 $1\,475$ & 1.5768 e-12 \\
 $1\,500$ & 1.5763 e-12 \\
 $1\,525$ & 1.5759 e-12 \\
 $1\,550$ & 1.5754 e-12 \\
 $1\,575$ & 1.5750 e-12 \\
\hline
\end{tabular}
\hfil
\begin{tabular}{|c|c|}
\hline
$\log(x_0)$ & $\varepsilon_{\theta,num}(x_0) $ \\
\hline
 $1\,600$ & 1.5745 e-12 \\
 $1\,625$ & 1.5741 e-12 \\
 $1\,650$ & 1.5737 e-12 \\
 $1\,675$ & 1.5733 e-12 \\
 $1\,700$ & 1.5730 e-12 \\
 $1\,725$ & 1.5726 e-12 \\
 $1\,750$ & 1.5722 e-12 \\
 $1\,775$ & 1.5719 e-12 \\
 $1\,800$ & 1.5716 e-12 \\
 $1\,825$ & 1.5712 e-12 \\
 $1\,850$ & 1.5709 e-12 \\
 $1\,875$ & 1.5706 e-12 \\
 $1\,900$ & 1.5703 e-12 \\
 $1\,925$ & 1.5700 e-12 \\
 $1\,950$ & 1.5697 e-12 \\
 $1\,975$ & 1.5695 e-12 \\
 $2\,000$ & 1.5692 e-12 \\
 $2\,025$ & 1.5689 e-12 \\
 $2\,050$ & 1.5687 e-12 \\
 $2\,075$ & 1.5485 e-12 \\
 $2\,100$ & 1.3246 e-12 \\
 $2\,125$ & 1.1333 e-12 \\
 $2\,150$ & 9.6943 e-13 \\
 $2\,175$ & 8.2952 e-13 \\
 $2\,200$ & 7.0974 e-13 \\
 $2\,225$ & 6.0730 e-13 \\
 $2\,250$ & 5.1965 e-13 \\
 $2\,275$ & 4.4472 e-13 \\
 $2\,300$ & 3.8058 e-13 \\
 $2\,325$ & 3.2572 e-13 \\
 $2\,350$ & 2.7879 e-13 \\
 $2\,375$ & 2.3863 e-13 \\
 $2\,400$ & 2.0426 e-13 \\
 $2\,425$ & 1.7484 e-13 \\
 $2\,450$ & 1.4967 e-13 \\
 $2\,475$ & 1.2814 e-13 \\
 $2\,500$ & 1.0972 e-13 \\
 $2\,525$ & 9.3932 e-14 \\
 $2\,550$ & 8.0424 e-14 \\
 $2\,575$ & 6.8870 e-14 \\
 $2\,600$ & 5.8977 e-14 \\
 $2\,625$ & 5.0505 e-14 \\
 $2\,650$ & 4.3255 e-14 \\
 $2\,675$ & 3.7045 e-14 \\
 $2\,700$ & 3.1729 e-14 \\
 $2\,725$ & 2.7178 e-14 \\
 $2\,750$ & 2.3282 e-14 \\
 $2\,775$ & 1.9945 e-14 \\
 $2\,800$ & 1.7087 e-14 \\
 $2\,825$ & 1.4639 e-14 \\
 $2\,850$ & 1.2543 e-14 \\
\hline
\end{tabular}
\hfil
\begin{tabular}{|c|c|}
\hline
$\log(x_0)$ & $\varepsilon_{\theta,num}(x_0) $ \\
\hline
 $2\,875$ & 1.0748 e-14 \\
 $2\,900$ & 9.2090 e-15 \\
 $2\,925$ & 7.8910 e-15 \\
 $2\,950$ & 6.7624 e-15 \\
 $2\,975$ & 5.7958 e-15 \\
 $3\,000$ & 4.9678 e-15 \\
 $3\,100$ & 2.6823 e-15 \\
 $3\,200$ & 1.4496 e-15 \\
 $3\,300$ & 7.8431 e-16 \\
 $3\,400$ & 4.2480 e-16 \\
 $3\,500$ & 2.3037 e-16 \\
 $3\,600$ & 1.2507 e-16 \\
 $3\,700$ & 6.7993 e-17 \\
 $3\,800$ & 3.7015 e-17 \\
 $3\,900$ & 2.0181 e-17 \\
 $4\,000$ & 1.1021 e-17 \\
 $4\,100$ & 6.0283 e-18 \\
 $4\,200$ & 3.3037 e-18 \\
 $4\,300$ & 1.8141 e-18 \\
 $4\,400$ & 9.9819 e-19 \\
 $4\,500$ & 5.5050 e-19 \\
 $4\,600$ & 3.0433 e-19 \\
 $4\,700$ & 1.6899 e-19 \\
 $4\,800$ & 9.4381 e-20 \\
 $4\,900$ & 5.3014 e-20 \\
 $5\,000$ & 2.9942 e-20 \\
 $6\,000$ & 1.2976 e-22 \\
 $7\,000$ & 8.5161 e-25 \\
 $8\,000$ & 7.7852 e-27 \\
 $9\,000$ & 9.2215 e-29 \\
 $10\,000$ & 1.3680 e-30 \\
 $20\,000$ & 1.9347 e-45 \\
 $30\,000$ & 6.6587 e-57 \\
 $40\,000$ & 1.3469 e-66 \\
 $50\,000$ & 3.7291 e-75 \\
 $60\,000$ & 6.6646 e-83 \\
 $70\,000$ & 4.9111 e-90 \\
 $80\,000$ & 1.1133 e-96 \\
 $90\,000$ & 6.3304 e-103 \\
 $100\,000$ & 7.7824 e-109 \\
 $200\,000$ & 1.2375 e-156 \\
 $300\,000$ & 2.1902 e-193 \\
 $400\,000$ & 2.1118 e-224 \\
 $500\,000$ & 9.5685 e-252 \\
 $600\,000$ & 1.7723 e-276 \\
 $700\,000$ & 3.1360 e-299 \\
 $800\,000$ & 2.0568 e-320 \\
 $900\,000$ & 2.5885 e-340 \\
 $10^6$ & 3.8635 e-359 \\
 $10^7$ & 1.0364 e-1153 \\
 $10^8$ & 1.7060 e-3669 \\
\hline
\end{tabular}
}
\end{table}

\begin{table}[h!]
\caption{Values for $\varepsilon_{\pi,num}(x_1)$ are calculated using Corollary \ref{cor_num_pi}, with Theorem \ref{prop_num_pi}.~ 
Note that here $x_0=10^{15}$ and that our sets $\{b_i\}_{i=1}^N$ and $\{b_i'\}_{i=1}^M$ are more refined than as provided by Tables \ref{table:theta-small}, \ref{table:theta-med} and \ref{table:theta}.}
\label{tab:pi}
\centering
\resizebox{!}{0.47\textheight}
{
\begin{tabular}{|c|c|}
\hline
$\log(x_1)$ & $\varepsilon_{\pi,num}(x_1)$ \\
\hline
 $44$ & 1.7893 e-8 \\
 $45$ & 1.1449 e-8 \\
 $46$ & 7.2959 e-9 \\
 $47$ & 4.6388 e-9 \\
 $48$ & 2.9451 e-9 \\
 $49$ & 1.8680 e-9 \\
 $50$ & 1.1785 e-9 \\
 $51$ & 7.4479 e-10 \\
 $52$ & 4.7046 e-10 \\
 $53$ & 2.9707 e-10 \\
 $54$ & 1.8753 e-10 \\
 $55$ & 1.1785 e-10 \\
 $56$ & 7.4191 e-11 \\
 $57$ & 4.6692 e-11 \\
 $58$ & 2.9380 e-11 \\
 $59$ & 1.8774 e-11 \\
 $60$ & 1.2330 e-11 \\
 $61$ & 8.4134 e-12 \\
 $62$ & 6.0325 e-12 \\
 $63$ & 4.5827 e-12 \\
 $64$ & 3.6978 e-12 \\
 $65$ & 3.1557 e-12 \\
 $66$ & 2.8216 e-12 \\
 $67$ & 2.6138 e-12 \\
 $68$ & 2.4828 e-12 \\
 $69$ & 2.3985 e-12 \\
 $70$ & 2.3427 e-12 \\
 $71$ & 2.3043 e-12 \\
 $72$ & 2.2766 e-12 \\
 $73$ & 2.2555 e-12 \\
 $74$ & 2.2387 e-12 \\
 $75$ & 2.2244 e-12 \\
 $76$ & 2.2120 e-12 \\
 $77$ & 2.2006 e-12 \\
 $78$ & 2.1901 e-12 \\
 $79$ & 2.1802 e-12 \\
 $80$ & 2.1708 e-12 \\
 $81$ & 2.1617 e-12 \\
 $82$ & 2.1530 e-12 \\
 $83$ & 2.1446 e-12 \\
 $84$ & 2.1364 e-12 \\
 $85$ & 2.1284 e-12 \\
 $86$ & 2.1207 e-12 \\
 $87$ & 2.1132 e-12 \\
 $88$ & 2.1059 e-12 \\
 $89$ & 2.0988 e-12 \\
 $90$ & 2.0919 e-12 \\
 $91$ & 2.0851 e-12 \\
 $92$ & 2.0786 e-12 \\
 $93$ & 2.0721 e-12 \\
 $94$ & 2.0659 e-12 \\
 $95$ & 2.0598 e-12 \\
\hline
\end{tabular}
\hfil
\begin{tabular}{|c|c|}
\hline
$\log(x_1)$ &  $\epsilon_{\pi,num}(x_1)$ \\
\hline
 $96$ & 2.0538 e-12 \\
 $97$ & 2.0480 e-12 \\
 $98$ & 2.0423 e-12 \\
 $99$ & 2.0367 e-12 \\
 $100$ & 2.0339 e-12 \\
 $110$ & 1.9853 e-12 \\
 $120$ & 1.9457 e-12 \\
 $130$ & 1.9126 e-12 \\
 $140$ & 1.8847 e-12 \\
 $150$ & 1.8608 e-12 \\
 $160$ & 1.8401 e-12 \\
 $170$ & 1.8219 e-12 \\
 $180$ & 1.8059 e-12 \\
 $190$ & 1.7917 e-12 \\
 $200$ & 1.7789 e-12 \\
 $210$ & 1.7675 e-12 \\
 $220$ & 1.7571 e-12 \\
 $230$ & 1.7476 e-12 \\
 $240$ & 1.7390 e-12 \\
 $250$ & 1.7311 e-12 \\
 $260$ & 1.7238 e-12 \\
 $270$ & 1.7171 e-12 \\
 $280$ & 1.7108 e-12 \\
 $290$ & 1.7051 e-12 \\
 $300$ & 1.6997 e-12 \\
 $310$ & 1.6946 e-12 \\
 $320$ & 1.6899 e-12 \\
 $330$ & 1.6855 e-12 \\
 $340$ & 1.6814 e-12 \\
 $350$ & 1.6775 e-12 \\
 $360$ & 1.6738 e-12 \\
 $370$ & 1.6703 e-12 \\
 $380$ & 1.6670 e-12 \\
 $390$ & 1.6639 e-12 \\
 $400$ & 1.6609 e-12 \\
 $410$ & 1.6581 e-12 \\
 $420$ & 1.6554 e-12 \\
 $430$ & 1.6529 e-12 \\
 $440$ & 1.6505 e-12 \\
 $450$ & 1.6481 e-12 \\
 $475$ & 1.6428 e-12 \\
 $500$ & 1.6380 e-12 \\
 $525$ & 1.6336 e-12 \\
 $550$ & 1.6296 e-12 \\
 $575$ & 1.6260 e-12 \\
 $600$ & 1.6227 e-12 \\
 $625$ & 1.6197 e-12 \\
 $650$ & 1.6169 e-12 \\
 $675$ & 1.6143 e-12 \\
 $700$ & 1.6119 e-12 \\
 $725$ & 1.6097 e-12 \\
 $750$ & 1.6076 e-12 \\
\hline
\end{tabular}
\hfil
\begin{tabular}{|c|c|}
\hline
$\log(x_1)$ &  $\epsilon_{\pi,num}(x_1)$ \\
\hline
 $775$ & 1.6057 e-12 \\
 $800$ & 1.6038 e-12 \\
 $825$ & 1.6021 e-12 \\
 $850$ & 1.6005 e-12 \\
 $875$ & 1.5990 e-12 \\
 $900$ & 1.5976 e-12 \\
 $925$ & 1.5962 e-12 \\
 $950$ & 1.5949 e-12 \\
 $975$ & 1.5937 e-12 \\
 $1\,000$ & 1.5925 e-12 \\
 $1\,025$ & 1.5914 e-12 \\
 $1\,050$ & 1.5904 e-12 \\
 $1\,075$ & 1.5894 e-12 \\
 $1\,100$ & 1.5885 e-12 \\
 $1\,125$ & 1.5875 e-12 \\
 $1\,150$ & 1.5867 e-12 \\
 $1\,175$ & 1.5858 e-12 \\
 $1\,200$ & 1.5850 e-12 \\
 $1\,225$ & 1.5843 e-12 \\
 $1\,250$ & 1.5836 e-12 \\
 $1\,275$ & 1.5828 e-12 \\
 $1\,300$ & 1.5822 e-12 \\
 $1\,325$ & 1.5815 e-12 \\
 $1\,350$ & 1.5809 e-12 \\
 $1\,375$ & 1.5803 e-12 \\
 $1\,400$ & 1.5797 e-12 \\
 $1\,425$ & 1.5791 e-12 \\
 $1\,450$ & 1.5786 e-12 \\
 $1\,475$ & 1.5781 e-12 \\
 $1\,500$ & 1.5776 e-12 \\
 $1\,525$ & 1.5771 e-12 \\
 $1\,550$ & 1.5766 e-12 \\
 $1\,575$ & 1.5761 e-12 \\
 $1\,600$ & 1.5757 e-12 \\
 $1\,625$ & 1.5753 e-12 \\
 $1\,650$ & 1.5749 e-12 \\
 $1\,675$ & 1.5745 e-12 \\
 $1\,700$ & 1.5741 e-12 \\
 $1\,725$ & 1.5737 e-12 \\
 $1\,750$ & 1.5733 e-12 \\
 $1\,775$ & 1.5729 e-12 \\
 $1\,800$ & 1.5726 e-12 \\
 $1\,825$ & 1.5723 e-12 \\
 $1\,850$ & 1.5719 e-12 \\
 $1\,875$ & 1.5716 e-12 \\
 $1\,900$ & 1.5713 e-12 \\
 $1\,925$ & 1.5710 e-12 \\
 $1\,950$ & 1.5707 e-12 \\
 $1\,975$ & 1.5704 e-12 \\
 $2\,000$ & 1.5701 e-12 \\
 $2\,100$ & 1.3254 e-12 \\
 $2\,200$ & 7.1013 e-13 \\
\hline
\end{tabular}
\hfil
\begin{tabular}{|c|c|}
\hline
$\log(x_1)$ &  $\epsilon_{\pi,num}(x_1)$ \\
\hline
 $2\,300$ & 3.8078 e-13 \\
 $2\,400$ & 2.0436 e-13 \\
 $2\,500$ & 1.0977 e-13 \\
 $2\,600$ & 5.9004 e-14 \\
 $2\,700$ & 3.1743 e-14 \\
 $2\,800$ & 1.7095 e-14 \\
 $2\,900$ & 9.2127 e-15 \\
 $3\,000$ & 4.9698 e-15 \\
 $3\,100$ & 2.6833 e-15 \\
 $3\,200$ & 1.4502 e-15 \\
 $3\,300$ & 7.8459 e-16 \\
 $3\,400$ & 4.2495 e-16 \\
 $3\,500$ & 2.3044 e-16 \\
 $3\,600$ & 1.2511 e-16 \\
 $3\,700$ & 6.8015 e-17 \\
 $3\,800$ & 3.7027 e-17 \\
 $3\,900$ & 2.0187 e-17 \\
 $4\,000$ & 1.1024 e-17 \\
 $4\,100$ & 6.0301 e-18 \\
 $4\,200$ & 3.3046 e-18 \\
 $4\,300$ & 1.8146 e-18 \\
 $4\,400$ & 9.9846 e-19 \\
 $4\,500$ & 5.5065 e-19 \\
 $4\,600$ & 3.0441 e-19 \\
 $4\,700$ & 1.6903 e-19 \\
 $4\,800$ & 9.4404 e-20 \\
 $4\,900$ & 5.3026 e-20 \\
 $5\,000$ & 2.9949 e-20 \\
 $6\,000$ & 1.2979 e-22 \\
 $7\,000$ & 8.5175 e-25 \\
 $8\,000$ & 7.7862 e-27 \\
 $9\,000$ & 9.2230 e-29 \\
 $10\,000$ & 1.3682 e-30 \\
 $20\,000$ & 1.9349 e-45 \\
 $30\,000$ & 6.6592 e-57 \\
 $40\,000$ & 1.3470 e-66 \\
 $50\,000$ & 3.7292 e-75 \\
 $60\,000$ & 6.6648 e-83 \\
 $70\,000$ & 4.9112 e-90 \\
 $80\,000$ & 1.1133 e-96 \\
 $90\,000$ & 6.3306 e-103 \\
 $100\,000$ & 7.7825 e-109 \\
 $200\,000$ & 1.2375 e-156 \\
 $300\,000$ & 2.1902 e-193 \\
 $400\,000$ & 2.1118 e-224 \\
 $500\,000$ & 9.5685 e-252 \\
 $600\,000$ & 1.7723 e-276 \\
 $700\,000$ & 3.1360 e-299 \\
 $800\,000$ & 2.0569 e-320 \\
 $900\,000$ & 2.5885 e-340 \\
 $10^6$ & 3.8635 e-359 \\
 $10^7$ & 1.0364 e-1153 \\
\hline
\end{tabular}
}
\end{table}

\begin{table}[h!]
\caption{Sample of values showing $ \varepsilon_{\pi,asymp}(x_1) $ interpolates an upper bound for $ \varepsilon_{\pi,num}(x_1)$ with $A_\pi = 121.107$, $B=3/2$, and $C=2$. See Corollary \ref{epsilon_asymp_pi_explicit}.
Note that values $ \varepsilon_{\pi,num}(x_1,\infty)$  displayed are computed using \eqref{mu_num_def} from Theorem \ref{prop_num_pi} rather than Corollary \ref{cor_num_pi}.}
\label{interp-table}
\centering
\begin{tabular}{|r|l|l|}
\hline
$\log(x_1)$ & $ \varepsilon_{\pi,asymp}(x_1) $ & $ \varepsilon_{\pi,num}(x_1,\infty)$ \\
\hline
 $100$ & 1.9202 & 2.0495 e-12 \\
 $1\,000$ & 6.6533 e-7 & 1.5938 e-12 \\
 $2\,000$ & 2.8341 e-11 & 1.5707 e-12 \\
 $3\,000$ & 1.0385 e-14 & 4.9711 e-15 \\
 $4\,000$ & 1.2145 e-17 & 1.1026 e-17 \\
 $5\,000$ & 3.0305 e-20 & 2.9954 e-20 \\
 $6\,000$ & 1.3052 e-22 & 1.2980 e-22 \\
 $7\,000$ & 8.5363 e-25 & 8.5185 e-25 \\
 $8\,000$ & 7.7910 e-27 & 7.7871 e-27 \\
 $9\,000$ & 9.3522 e-29 & 9.2236 e-29 \\
 $10\,000$ & 1.4137 e-30 & 1.3683 e-30 \\
 \hline
 \end{tabular}\quad
 \begin{tabular}{|r|l|l|}
 \hline
 $\log(x_0)$ & $ \varepsilon_{\pi,asymp}(x_1) $ & $ \varepsilon_{\pi,num}(x_1,\infty)$ \\
 \hline
 $11\,000$ & 2.6036 e-32 & 2.4758 e-32 \\
 $12\,000$ & 5.6934 e-34 & 5.3287 e-34 \\
 $13\,000$ & 1.4481 e-35 & 1.3361 e-35 \\
 $14\,000$ & 4.2127 e-37 & 3.8368 e-37 \\
 $15\,000$ & 1.3824 e-38 & 1.2443 e-38 \\
 $16\,000$ & 5.0581 e-40 & 4.5033 e-40 \\
 $17\,000$ & 2.0432 e-41 & 1.8009 e-41 \\
 $18\,000$ & 9.0354 e-43 & 7.8897 e-43 \\
 $19\,000$ & 4.3424 e-44 & 3.7589 e-44 \\
 $20\,000$ & 2.2536 e-45 & 1.9349 e-45 \\
 &&\\
\hline
\end{tabular}
\end{table}

\begin{table}[h!]
\caption{$E_{\pi}(x) \le  A_{\pi}\left( \frac{\log(x)}{R} \right)^B \exp\left( -C \sqrt{\frac{\log(x)}{R}} \right) \ \text{for all}\ x\ge x_0$. See Corollary \ref{cor:alt1}.
The bold row is the result from Corollary \ref{epsilon_asymp_pi_explicit}.}
\label{table:alt1}
\begin{tabular}{|cccc|}
\hline
$ A_{\pi}$ & $B$ & $C$ & $\log(x_0)$ \\
\hline
$0.000120$ & $0.25$ & $1.00$ &$ 22.955$\\ 
$0.826$ & $0.25$ & $1.00$ & $1.000$ \\
$1.41$ & $0.50$ & $1.50$ & $2.000$ \\
$1.76$ & $1.00$ & $1.50$ & $3.000$ \\
$2.22$ & $1.50$ & $1.50$ & $3.000$ \\
$12.4$ & $1.50$ & $1.90$ & $1.000$ \\
$38.8$ & $1.50$ & $1.95$ & $1.000$ \\
{\bf121.107} & {\bf1.50} & {\bf2.00} & {\bf1.000} \\
$6.60$ & $2.00$ & $2.00$ & $3.000$ \\
\hline
\end{tabular}
\end{table}

\begin{table}[h!]
\caption{Other forms of bounds $E_{\pi}\le \mathcal{B}(x)$. See Corollary \ref{cor:alt2}.}
\label{table:alt2}
\begin{tabular}{|r|l|}
\hline
$\mathcal{B}(x)$& Range \\
\hline
$2\log(x)x^{-1/2}$ & $1 \leq \log(x) \leq 57$  \\
$\log(x)^{3/2}x^{-1/2}$ & $1 \leq \log(x) \leq 65.65$  \\
$\tfrac{1}{8\pi}\log(x)^{2}x^{-1/2}$ & $8 \leq \log(x) \leq 60.8$  \\
$\log(x)^{2}x^{-1/2}$ & $1 \leq \log(x) \leq 70.6$  \\
$\log(x)^{3}x^{-1/2}$ & $1 \leq \log(x) \leq 80$  \\
$x^{-1/3}$ & $1 \leq \log(x) \leq 80.55$  \\
$x^{-1/4}$ & $1 \leq \log(x) \leq 107.6$  \\
$x^{-1/5}$ & $1 \leq \log(x) \leq 134.8$  \\
$x^{-1/10}$ & $1 \leq \log(x) \leq 270.8$  \\
$x^{-1/50}$ & $1 \leq \log(x) \leq 1358.6$  \\
$x^{-1/100}$ & $1 \leq \log(x) \leq 3757.6$  \\
\hline
\end{tabular}   
\end{table}

\end{small}

\end{document}